\documentclass[12pt]{article}
\usepackage{amsmath,amsthm,amssymb,latexsym,color}
\usepackage{hyperref}

\theoremstyle{plain}
\newtheorem{theor}{Theorem}[section]
\newtheorem*{theorA}{Theorem~A}
\newtheorem*{corB}{Corollary~B}

\newtheorem{prop}[theor]{Proposition}

\newtheorem{cor}[theor]{Corollary}
\newtheorem{lemma}[theor]{Lemma}
\newtheorem{rem}[theor]{Remark}
\theoremstyle{remark}

\def\R{{\mathbb R}}

\def\Event{{\mathcal E}}

\def\Prob{{\mathbf{P}}}
\def\Var{{\mathbf{Var}}}
\def\E{{\mathbf E}}
\def\M{{\mathbf{Med}}}
\def\G{\mathcal{G}}

\def\gsm{\gtrsim}
\def\x1{\xi}
\def\p0{\frac{2\log n}{\log(2e)}}

\title{The variance of the $\ell_p^n$--norm of the Gaussian vector, and Dvoretzky's theorem}
\author{Anna Lytova\footnote{University of Opole, Poland; email: alytova@math.uni.opole.pl. A significant part
of this work was done when A.L.\ was visiting Princeton University in January--February, 2017}\quad and\quad
Konstantin Tikhomirov\footnote{Princeton University, NJ; email: kt12@math.princeton.edu.
The research is partially supported by the Simons Foundation.}}

\begin{document}

\maketitle

\begin{abstract}
Let $n$ be a large integer, and let $G$ be the standard
Gaussian vector in $\R^n$.
Paouris, Valettas and Zinn (2015) showed that for all $p\in[1,c\log n]$,
the variance of the $\ell_p^n$--norm of $G$ is equivalent,
up to a constant multiple, to $\frac{2^p}{p}n^{2/p-1}$, and for
$p\in[C\log n,\infty]$, $\Var\|G\|_p\simeq (\log n)^{-1}$.
Here, $C,c>0$ are universal constants. That result left open the question of
estimating the variance for $p$ logarithmic in $n$.
In this note, we resolve the question by providing a complete characterization of $\Var\|G\|_p$
for all $p$. We show that there exist two transition points (windows) in which behavior of $\Var\|G\|_p$ significantly changes.
We also discuss some implications of our result in context of random Dvoretzky's theorem for $\ell_p^n$.
\end{abstract}



\section{Introduction}
\label{s:intro}

Let $n$ be a large integer, $p$ be a number in $[1,\infty]$, and denote by $\|\cdot\|_p$
the standard $\ell_p^n$--norm in $\R^n$. Let $G$ be the standard $n$-dimensional Gaussian vector.
Variance of the $\|\cdot\|_p$--norm of $G$ may serve as a basic example of the concentration of measure phenomenon
(most of the Gaussian mass is located in a thin shell of an appropriately rescaled $\ell_p^n$--ball).
It is well known that for a fixed $p<\infty$, $\Var\|G\|_p\simeq v_p n^{2/p-1}$, where the quantity $v_p$ depends only on $p$
and not on $n$
(see, in particular, \cite{Naor} and \cite{PVZ}), whereas
the variance of the $\|\cdot\|_\infty$--norm of $G$ is of order $(\log n)^{-1}$ (see, for example, \cite[p.~47--48]{Ch:13}
and \cite{PVZ}). At the same time, for $p$ growing to infinity with $n$, no sharp results were available
until quite recently. In \cite{PVZ}, Paouris, Valettas and Zinn showed that $\Var\|G\|_p\simeq \frac{2^p}{p}n^{2/p-1}$
for $p\leq c\log n$ and $\Var\|G\|_p\simeq (\log n)^{-1}$ for $p\geq C\log n$
($C,c>0$ being universal constants). This result of \cite{PVZ} leaves the gap $c\log n\leq p\leq C\log n$
in which the behavior of the variance was not clarified. The authors of \cite{PVZ} conjectured
that the variance changes from polynomially small in $n$ to logarithmic around $\widetilde p=\log_2 (n)$.
This conjecture was the starting point of our work.

The question of computing the Gaussian variance of the $\ell_p^n$--norm seems natural on its own right;
nevertheless, it gains more sense in the context of asymptotic geometric analysis.
Since the fundamental discovery of Milman \cite{Milman},
it is known that Gaussian concentration properties of a norm $\|\cdot\|$ in $\R^n$
are strongly connected
with geometry of random subspaces of $(\R^n,\|\cdot\|)$. The classical theorem of Dvoretzky \cite{Dvoretzky}
asserts that every infinite-dimensional Banach space contains finite subspaces of arbitrarily
large dimension which are arbitrarily close to Euclidean (in the Banach--Mazur metric).
Milman showed in \cite{Milman} that a stronger result takes place.
Given a norm $\|\cdot\|$ in $\R^n$, a subspace
$E\subset\R^n$ and a real number $K\geq 1$,  we will (rather, unconventionally) call the subspace
{\it $K$-spherical} if
$\sup_{x\in E,\|x\|_2=1}\|x\|/\inf_{x\in E,\|x\|_2=1}\|x\|\leq K$.
The theorem of Milman states that
for any norm $\|\cdot\|$ in $\R^n$
with the Lipschitz constant $L$ and any $\varepsilon\in(0,1/2)$, the random $\frac{c\varepsilon^2}{\log(1/\varepsilon)}
\big(\frac{\E \|G\|}{L}\big)^2$--dimensional subspace of $(\R^n,\|\cdot\|)$ with uniform (rotation-invariant) distribution
is $(1+\varepsilon)$--spherical with probability close to $1$.
In particular, the Dvoretzky--Rogers lemma implies that for any norm $\|\cdot\|$ with the unit ball
in {\it John's position}, the random $\frac{c\varepsilon^2\log n}{\log(1/\varepsilon)}$--dimensional
subspace is $(1+\varepsilon)$--spherical with large probability.
We refer to monographs and surveys \cite{MS 1986,Pisier,Schechtman 2013,AGM 2015}
for more information as well as to papers \cite{PVZ, PV 16a, PV 16b, T 2017, PV 17}
for some recent developments of the subject.
In this text, we leave out any discussion of the {\it existential} Dvoretzky theorem
which is concerned with finding at least one large almost Euclidean subspace (the best known general
result in this direction is due to Schechtman \cite{S 2006}) as well as the {\it isomorphic} Dvoretzky theorem
which deals with the regime when distortion $\varepsilon$ grows to infinity with $n$ (see, in particular, \cite{MS iso}).

In the regime of ``constant distortion'' (say, when $1+\varepsilon=2$) the result of Milman
is sharp, that is, if a random $k$-dimensional subspace is $2$--spherical with high probability
then necessarily $k\leq C\big(\frac{\E\|G\|}{L}\big)^2$
(see Milman--Schechtman \cite{MS 1997} and Huang--Wei \cite{HW} for reverse estimates matching Milman's bound).
However, when $\varepsilon$ tends to zero with $n\to\infty$, the original estimate is suboptimal.
Gordon \cite{Gordon 1985} and later Schechtman \cite{Schechtman 1989}
improved the dependence on $\varepsilon$ from $\frac{\varepsilon^2}{\log(1/\varepsilon)}$
to $\varepsilon^2$, which is sharp for {\it some} norms but not in general.
For example, it was shown in \cite{Schechtman 2007} and \cite{KT:13}
that a random $k$--dimensional subspace of $\ell_\infty^n$ is
$(1+\varepsilon)$--spherical with probability close to one if and only if
$k\lesssim \frac{\varepsilon\log n}{\log(1/\varepsilon)}$. Moreover, for $1$-unconditional norms
in {\it the $\ell$-position}, it was proved in \cite{T 2017} that random $\frac{c\varepsilon\log n}{\log(1/\varepsilon)}$--dimensional
subspaces are $(1+\varepsilon)$--spherical with high probability.
For arbitrary norms, the problem of interdependence between $\varepsilon$ and the dimension in the random Dvoretzky theorem
is wide open, and even in the class of $\ell_p^n$--spaces there is no complete solution as of this writing.

A considerable progress in estimating the distortion (in the ``almost isometric'' regime)
of uniform random subspaces of $\ell_p^n$
for all $p$ was due to Naor \cite{Naor}
and Paouris, Valettas and Zinn \cite{PVZ}.
For a fixed $2<p<\infty$, Naor \cite{Naor} obtained concentration inequalities which,
in particular, can be employed to show that random $u_{n,p}(\varepsilon n)^{2/p}$--dimensional sections of the
$\ell_p^n$--ball are $(1+\varepsilon)$--spherical with probability close to one whenever $\varepsilon\geq n^{-v_p}$
(where $v_p>0$ depends only on $p$ and $u_{n,p}$ is a quantity of order polylogarithmic in $n$
arising from the application of the covering argument).
The bound $w_p(\varepsilon n)^{2/p}$ on the dimension of typical
$(1+\varepsilon)$-Euclidean subspaces of $\ell_p^n$ (for $w_p>0$ depending only on $p>2$) was
confirmed by Paouris, Valettas and Zinn \cite{PVZ} in the range $\varepsilon\geq n^{-v_p}$, and
it was shown that for a fixed $p$ the estimate is close to optimal.
The paper \cite{PVZ} provides bounds (upper and lower) for the Dvoretzky dimension,
as well as concentration inequalities for the standard Gaussian vector and the Gaussian variance in different regimes
giving an emphasis to the case when $p$ grows with $n$.
However, for $p$ logarithmic in $n$, the results are not sharp.

In the context of Dvoretzky's theorem, the $\ell_p^n$--spaces for logarithmic $p$
supply rather interesting geometric examples. As was observed in \cite{PVZ}, there
are universal constants $c,C>0$ such that, say, $\Var\|G\|_{c\log n}\leq n^{-1/2}$,
whereas $\Var\|G\|_{C\log n}\gtrsim \frac{1}{\log n}$; thus, the variance can be quite sensitive
to replacing a norm with an equivalent norm.
Note that the bounds for the variance immediately imply that, for example,
the random $3$-dimensional subspace of $\ell_{c\log n}^n$ is $(1+n^{-c})$--spherical
with probability at least $1-n^{-c}$ for a universal constant $c>0$ (and instead of $3$ we can
take any constant dimension). At the same time, most of $3$--dimensional subspaces of
$\ell_\infty^n$ (which is a constant Banach--Mazur distance away from $\ell_{c\log n}^n$)
are not even $(1+\frac{1}{\log n})$--spherical \cite{KT:13}.
The result of \cite{PVZ} leaves open the question whether there is a ``phase transition'' point $\widetilde p=\widetilde p(n)$
such that for any $\delta>0$ and all sufficiently large $n$ we have $\Var\|G\|_{(1-\delta)\widetilde p}\leq n^{-v_\delta}$
and $\Var\|G\|_{(1+\delta)\widetilde p}\geq \frac{v_\delta}{\log n}$, where $v_\delta>0$ depends only on $\delta$.
Our result answers this question and completely settles the problem of computing the Gaussian variance
of $\|\cdot\|_p$--norms. Below, for any two quantities $a,b$ we write ``$a\simeq b$''
if $C^{-1}a\leq b\leq Ca$ for a universal constant $C>0$.

\begin{theorA}
There is a universal constant $n_0>0$ with the following property. Let $n\geq n_0$ and let $G$ be the standard Gaussian
vector in $\R^n$.
Further, denote by $\xi$ the quantile of order $1-\frac{1}{n}$ with respect to the distribution of the absolute value of
a standard Gaussian variable $|g|$, i.e.\ such that $\Prob\{|g|\leq \xi\}=1-\frac{1}{n}$. Then
\begin{itemize}
\item For all $p$ in the range $1\leq p\leq \frac{2\log n}{\log(2e)}$ we have
$$
\Var(\|G\|_p)\simeq \frac{2^p}{p}n^{2/p-1};
$$
\item For $\frac{2\log n}{\log(2e)}\leq p\leq \xi^2$ we have
$$
\Var(\|G\|_p)\simeq \frac{\exp\big(-\frac{p}{2e}n^{2/p} +\log n\big)}
{\sqrt{\log n}\,\,\big(\sqrt{\log n}+p-\frac{2\log n}{\log (2e)}\big)};
$$
\item For $p\geq \xi^2$ we have
$$\Var(\|G\|_p)\simeq \frac{1}{\log n}\Big(1-\frac{\xi^2-\xi}{p}\Big).$$
\end{itemize}
\end{theorA}
Everywhere in this note, ``$\log$'' stands for the natural logarithm.
As we mentioned before, the above estimates in the regimes $p\leq c\log n$ and $p\geq C\log n$
were previously derived in \cite{PVZ}.
The variance of $\|G\|_p$, the way we represent it, is a piece-wise function,
with the pieces equivalent at respective boundary points.
The points $\frac{2\log n}{\log(2e)}$ and $\xi^2=2\log n-o(\sqrt{\log n})$ (see (\ref{eq:quantile}))
are chosen rather arbitrarily in a sense that each one can be shifted to the right or to the left by a small constant multiple
of $\sqrt{\log n}$, which would change the estimates only by a multiplicative constant.
In this connection, we prefer to speak about ``transition windows'' rather than ``transition points''.

To have a better picture of how the variance changes with $p$, it may be useful to consider its logarithm
$\log\Var\|G\|_p$ in the range $c\log n\leq p<\infty$ for some fixed small constant $c$, so that the term $n^{2/p}$
is bounded.
For $c\log n\leq p\leq\frac{2\log n}{\log(2e)}$,  $\log\Var\|G\|_p$ is an almost linear function of $p$.
In the range $p\geq \xi^2$, $\log\Var\|G\|_p$
is essentially of order $-\log\log n$ (up to a bounded multiple).
In the intermediate regime $\frac{2\log n}{\log(2e)}\leq p\leq \xi^2$, disregarding additive terms double logarithmic in $n$,
$\log\Var\|G\|_p$ behaves as $-\frac{p}{2e}n^{2/p}+\log n$, 
which is a convex function close to parabola $-\frac{(2\log n-p)^2}{4\log n}$ near the point $2\log n$.

Our result implies, in particular, that for any fixed $\delta\in(0,1)$ and all sufficiently large $n$, we have
$\Var\|G\|_{(2-\delta)\log n}\leq n^{-v_\delta}$ whereas $\Var\|G\|_{(2+\delta)\log n}\geq \frac{v_\delta}{\log n}$
for some $v_\delta>0$ depending only on $\delta$.
Observe that the Banach--Mazur distance between $\ell_{(2-\delta)\log n}^n$
and $\ell_{(2+\delta)\log n}^n$ is of order $1+O(\delta)$, so the
``power of $n$ to logarithmic'' transition happens at an almost isometric scale.
In the context of the random Dvoretzky theorem, this implies
\begin{corB}
For any $\delta\in (0,1)$,
there are $w_{\delta},n_{\delta}>0$ depending on $\delta$ with the following property.
For any $\varepsilon\in(0,1)$ and $n\geq n_{\delta}$, let
$2\le k\le \lceil w_\delta \log n/\log(2/\varepsilon)\rceil$,
and let $E$ be a uniformly distributed random $k$-dimensional subspace of $\R^n$.
Then
\begin{equation}
 \Prob\big\{E\mbox{ is $(1+\varepsilon)$--spherical subspace of }\ell_{(2-\delta)\log n}^n\big\}\geq 1-n^{-w_{\delta}}.
 \label{eq:spherical}
\end{equation}
At the same time,
\begin{equation}
 \Prob\big\{E\mbox{ is not $(1+\frac{w_{\delta}}{\log n})$--spherical subspace of }\ell_{(2+\delta)\log n}^n\big\}\geq w_\delta.
 \label{eq:not-spherical}
\end{equation}
\end{corB}

Our result highlights an interesting characteristic of the order statistics of the standard Gaussian vector $G=(g_1,g_2,\dots,g_n)$.
Let $(g_1^*, g_2^*,\dots,g_n^*)$ be the non-increasing rearrangement of the vector of absolute values $(|g_1|,|g_2|,\dots,|g_n|)$.
Then  Chernoff--type estimates imply that order statistics $g_i^*$ for $i$ relatively large,
say, at least a positive constant power of $n$, are strongly concentrated, so
that their typical fluctuations are small (at most a negative constant power of $n$).
Thus, the large (logarithmic in $n$) fluctuation of $\|G\|_p$ for $p\geq (2+\delta)\log n$
is due to the fact that the $p$-th powers of the first few order statistics comprise a relatively large portion of the
sum $\|G\|_p^p=\sum_{i=1}^n (g_i^*)^p$ with a significant probability, whereas for $p\leq (2-\delta)\log n$
the $p$-th powers of the first order statistics are typically hugely dominated by the total sum $\|G\|_p^p$.

\bigskip

Our technique of proving Theorem~A
is in certain aspects similar to \cite{PVZ}.
As in \cite{PVZ}, a crucial role in our argument is played by Talagrand's $L_1-L_2$ bound
(see Theorem~\ref{t:talagrand} in the next section), which allows to get sharper estimates
for the variance than the Poincar\'e inequality.
Another important step, also presented in \cite{PVZ}, consists in obtaining strong
upper bounds for negative moments of $\ell_p^n$--norms.
Our approach to bounding the moments is completely different from the one used in \cite{PVZ},
as, instead of relying on general Gaussian inequalities, we employ
a rather elementary but efficient technique involving lower deviation estimates for
the order statistics of random vectors. This allows us to get strong estimates including the case $p\approx 2\log n$
i.e.\ in the range not treated in \cite{PVZ}.
A principal new ingredient to our proof, compared to \cite{PVZ}, is the use of truncated Gaussians.
For a number $M>0$, we consider an auxiliary
function
$$f(G):=\bigg(\sum_{i=1}^n\min(M,|g_i|)^p\bigg)^{1/p},$$
and use a trivial inequality $\Var\|G\|_p\leq 2\E(\|G\|_p-f(G))^2+2\Var f(G)$.
It turns out that, for a carefully chosen truncation level $M$ (the right choice is not straightforward),
both terms in the last inequality can be estimated in an optimal way by combining Talagrand's
$L_1-L_2$ theorem with rather elementary
probabilistic arguments and bounds for truncated moments of Gaussian variables.
The truncation technique is also used to obtain matching lower bounds for the variance.
We will discuss this approach in more detail at the beginning of Section~\ref{s:upper}.

The organization of the rest of the paper is as follows. In Section~\ref{s:prelim}
we discuss notation and state several facts important for our work
as well as provide a detailed derivation of upper and lower bounds for truncated Gaussian moments (Section \ref{ss:trunc moment})
and lower deviation estimates for Gaussian order statistics, using Chernoff's inequality (Section \ref{ss:chernoff}).
In Section~\ref{s:negative} we provide upper bounds for the negative moments
of $\ell_p^n$--norms in terms of quantiles of the Gaussian distribution.
In Section~\ref{s:upper} we obtain upper bounds for the variance, and in Section~\ref{s:lower}
derive matching lower bounds.
For reader's convenience, we give a proof of Corollary~B in Section~\ref{s:corB}.

\section{Preliminaries}
\label{s:prelim}

Let us start with notation and some basic facts that will be useful for us.
The canonical inner product in $\R^n$ is denoted by $\langle \cdot,\cdot\rangle$.
Given a vector
$x=(x_1,x_2\dots,x_n)\in\R^n$ and a real number $1\leq p<\infty$, the standard $\ell_p^n$--norm of $x$ is defined as
$$
\|x\|_p:=\Big(\sum_{i=1}^n|x_i|^p\Big)^{1/p}.
$$
Additionally, the $\ell_\infty^n$--norm $\|x\|_\infty:=\max_{i\leq n}|x_i|$.
The following relation is true for any $x\in\R^n$:
\begin{equation}\label{eq:pq}
\|x\|_q\leq\|x\|_p\leq n^{1/p-1/q}\|x\|_q,\quad 1\le p\le q\leq\infty.
\end{equation}

Given a real number $t$, $\lfloor t\rfloor$ is the largest integer not exceeding $t$.
Universal constants are denoted by $C,c,\widetilde c$, etc.\
and their value may be different on different occasions.
Given two quantities $a$ and $b$, we write $a\simeq b$ whenever there is a universal constant $C\neq 0$
with $C^{-1} a\leq b\leq C a$. Further, for two non-negative quantities $a,b$ we write $a\lesssim b$
($a\gtrsim b$) if there is a universal constant $C>0$ with $a\leq Cb$ (respectively, $Ca\geq b$).
Sometimes it will be convenient for us to write the relation $a\lesssim b$ as $a=O(b)$.

\bigskip

The expectation of a random variable $Z$ will be denoted by $\E Z$, the variance --- by $\Var Z$,
and the median --- by $\M Z$.
Given an event $\Event$, by $\chi_{\Event}$ we denote the indicator function of $\Event$.
Throughout the text, standard Gaussian variables will be denoted by $g,g_1,g_2,\dots$
and the standard Gaussian vector in $\R^n$ --- by $G$.
It is well known (see, for example, \cite[Chapter~7]{F:68}) that the Gaussian distribution satisfies the relations
\begin{equation}\label{eq:normal approx}
\sqrt{\frac{2}{\pi}}\Big(\frac{1}{t}-\frac{1}{t^3}\Big)
e^{-t^2/2}<\Prob\big\{|g|\geq t\big\}< \sqrt{\frac{2}{\pi}}\frac{1}{t}e^{-t^2/2},\quad\quad t>0.
\end{equation}
The absolute moments of a standard Gaussian variable are given by
\begin{equation}\label{eq:gaussian moments}
\E|g|^p=\frac{1}{\sqrt{\pi}}2^{p/2}\,\Gamma\Big(\frac{p+1}{2}\Big),\quad\quad p>-1.
\end{equation}

The next theorem is a basis for our analysis; its ``discrete'' version was proved by M.~Talagrand in \cite{Talagrand}.
\begin{theor}[{Talagrand's $L_1\text{--}L_2$ bound; see
\cite{CEL}, \cite[Chapter~5]{Ch:13}}]\label{t:talagrand}
Suppose $f$ is an absolutely continuous function in $\R^n$ and let $\partial_i f$ ($i\leq n$)
be the partial derivatives of $f$.
Then we have
$$\Var\big(f(G)\big)\leq C\sum\limits_{i=1}^n \frac{\E|\partial_i f(G)|^2}
{1+\log\big(\sqrt{\E|\partial_i f(G)|^2}/\E|\partial_i f(G)|\big)},$$
where $C>0$ is a universal constant.
\end{theor}

\subsection{Bounds for truncated moments of Gaussian variables}
\label{ss:trunc moment}

In this subsection, we derive rather elementary upper and lower bounds for high moments of random variables
of the form $|g|\chi_{\{|g|\leq a\}}$ and $\min(|g|,a)$ for a fixed $a>0$. The results presented here
are by no means new, but may be hard to locate in literature.
For reader's convenience, we provide proofs.

Let us start with a simple calculus lemma.
\begin{lemma}
\label{l:inc-gamma}
Fix $0<a<\infty$. Let $f$ be a positive log-concave function on $[0,a]$,
and let $x_{\max}\in[0,a]$ be a point of global maximum for
$f$. Define
\begin{align*}
&x_\ell:=\min\big\{y\in [0, x_{\max}]:\, f(y)\ge f(x_{\max})/2\big\},
\\
&x_r:=\max\big\{y\in [x_{\max}, a]:\, f(y)\ge f(x_{\max})/2\big\}.
\end{align*}
Then
$$
\frac{1}{2}(x_r-x_\ell)f(x_{\max})\le\int_0^a f(x)\,dx\le 2(x_r-x_\ell)f(x_{\max}).
$$
\end{lemma}

As a consequence of the above statement, we get
\begin{lemma}\label{l:incom gamma}
Let $q,a\geq 1$ be some real numbers.
\begin{itemize}
\item If $q\leq a^2$ then
$$(q/e)^{q/2}\lesssim\int_0^a x^q e^{-x^2/2}\,dx\lesssim (q/e)^{q/2};$$

\item If $q\geq a^2$ then
$$\frac{a^{q+1}e^{-a^2/2}}{a+q-a^2}\lesssim\int_0^a x^q e^{-x^2/2}\,dx\lesssim \frac{a^{q+1}e^{-a^2/2}}{a+q-a^2}.$$
\end{itemize}
\end{lemma}
\begin{proof}
It is easy to see that $\sqrt{q}$ is the point of global maximum of the log-concave function
$f(x):=x^q e^{-x^2/2}$ on $[0,\infty)$, and $f(\sqrt{q})=(q/e)^{q/2}$.
We will use the last lemma to evaluate the integrals.

\underline{The case $q\leq a^2$.}
We can assume without loss of generality that $q$ is large (greater than a large absolute constant).
To get the desired bound it is enough to show that $x_r-x_\ell\simeq 1$,
where $x_{r}>x_\ell$ are the two solutions of the equation $2f(x)=f(\sqrt{q})=(q/e)^{q/2}$.
We can rewrite the equation in the form
\begin{equation}\label{eq: aux09843t}
\log (z+1)+\frac{\log 4}{q}=z,\quad z={x^2}/{q}-1.
\end{equation}
Since $q$ is large, we can assume that all solutions of the last equation satisfy $z\in[-c,c]$
for a small constant $c>0$.
Then, using Taylor's expansion for the logarithm, we obtain
$$z\simeq\pm\frac{1}{\sqrt{q}},$$
for the two solutions of \eqref{eq: aux09843t}.
Hence,
$$
x_r^2-q\simeq \sqrt{q},\quad
x_\ell^2-q\simeq -\sqrt{q}.
$$
The result follows.

\medskip

\underline{The case $q\geq a^2$.} Let $x_\ell$, $x_r$, $x_{\max}$ be defined as in Lemma~\ref{l:inc-gamma}.
We have $x_{\max}=x_r=a$, and $f(x_{\max})=f(a)=a^q\exp(-a^2/2)$.
To get the desired bounds, it suffices to show that there exist constants $c_1$, $c_2$ such that
$$
\frac{c_1a}{a+q-a^2}\le a-x_\ell\le \frac{c_2a}{a+q-a^2},
$$
where $2{x_\ell}^q\exp(-{x_\ell}^2/2)=f(a)$, and then apply Lemma \ref{l:inc-gamma}.
We will rely on the fact that $f(x)$ is strictly increasing on $[0,a]$.

For a sufficiently small universal constant $\widetilde c>0$ we have
$\log(1-z)>-z-z^2$, $|z|<\widetilde c$.
Hence, for $x:=a-\frac{\widetilde c a}{a+q-a^2}$ we obtain
\begin{align*}
\log f(x)&=q\log x-\frac{x^2}{2}\\
&=q\log a+q \log \Big(1-\frac{a-x}{a}\Big)-\frac{a^2}{2}+a(a-x)-\frac{(a-x)^2}{2}\\
&\geq \log f(a)+q \Big(-\frac{a-x}{a}-\frac{(a-x)^2}{a^2}\Big)+a(a-x)-\frac{(a-x)^2}{2}\\
&\geq \log f(a)-\frac{\widetilde c\,(q-a^2)}{a+q-a^2}-\frac{\widetilde c^2 q}{(a+q-a^2)^2}-\frac{1}{2}\widetilde c^2\\
&>\log f(a)-\log 2,
\end{align*}
where in the last inequality we used the condition $q\geq a^2$ and the fact that $\widetilde c$ is small.
Thus, $f(x)\geq\frac{1}{2}f(a)$ whence $x_\ell\leq a-\frac{\widetilde c a}{a+q-a^2}$.

Now, choose $x:=a-\frac{\widetilde C a}{a+q-a^2}$, where $\widetilde C>0$
is a large enough universal constant (say, $\widetilde C=10$ definitely suffices).
If $x<0$ then obviously $x_\ell\geq x$, and we are done.
Otherwise, we use the trivial relation
$\log(1-z)\leq -z$, $z\in(-\infty,1)$, to obtain
\begin{align*}
\log f(x)&=q\log a+q \log \Big(1-\frac{a-x}{a}\Big)-\frac{a^2}{2}+a(a-x)-\frac{(a-x)^2}{2}\\
&\leq\log f(a)-\frac{(q-a^2)(a-x)}{a}-\frac{(a-x)^2}{2}\\
&= \log f(a)- \frac{\widetilde C\,(q-a^2)}{a+q-a^2}-\frac{\widetilde C^2}{2}\frac{a^2}{(a+q-a^2)^2}.
\end{align*}
When $q\leq a+a^2$, the last term is less than $-\frac{1}{8}\widetilde C^2$, whereas for $q\geq a+a^2$,
the second term is less than $-\frac{1}{2}\widetilde C$. In any case, we get $f(x)\leq \frac{1}{2}f(a)$,
whence $x_\ell\geq a-\frac{\widetilde C a}{a+q-a^2}$. The result follows.
\end{proof}

\begin{cor}\label{c:moments}
Let $q,a\geq 1$ be some real numbers.

(i) If $q\leq a^2$ then
$$
(q/e)^{q/2}\lesssim\E\big(|g|\chi_{\{|g|\leq a\}}\big)^q\leq\E\min\big(|g|,a\big)^q\lesssim (q/e)^{q/2};
$$

(ii) If $q\geq a^2$ then
$$
\frac{a^{q+1}e^{-a^2/2}}{a+q-a^2}\lesssim\E\big(|g|\chi_{\{|g|\leq a\}}\big)^q
\leq\E\min\big(|g|,a\big)^q\lesssim\frac{qa^{q-1}e^{-a^2/2}}{a+q-a^2}.
$$

(iii) In particular, if  $\tau q\leq a^2\le q$ for some $\tau\in(0,1)$, then
$$
\frac{a^{q+1}e^{-a^2/2}}{a+q-a^2}\lesssim\E\big(|g|\chi_{\{|g|\leq a\}}\big)^q
\leq\E\min\big(|g|,a\big)^q\lesssim \frac{1}{\tau}\frac{a^{q+1}e^{-a^2/2}}{a+q-a^2}.
$$
\end{cor}

\begin{proof}
Since
\begin{align*}
&\E\big(|g|\chi_{\{|g|\leq a\}}\big)^q=\sqrt{2/\pi}\int_0^a t^q e^{-t^2/2}\,dt\quad\mbox{and}\quad
\\
&\E\min\big(|g|,a\big)^q=\sqrt{2/\pi}\int_0^a t^q e^{-t^2/2}\,dt+\sqrt{2/\pi}a^q\int_a^\infty e^{-t^2/2}\,dt,
\end{align*}
then, applying \eqref{eq:normal approx}, we get
$$
\E\big(|g|\chi_{\{|g|\leq a\}}\big)^q\leq\E\min\big(|g|,a\big)^q\leq \E\big(|g|\chi_{\{|g|\leq a\}}\big)^q+\sqrt{2/\pi}a^{q-1}e^{-a^2/2}.
$$
This and the second part of Lemma \ref{l:incom gamma} yield the assertion for $q\geq a^2$ and for $\tau q\leq a^2\le q$. The case $q\leq a^2$ follows from the first part of Lemma \ref{l:incom gamma} and the fact that $\max_\R t^q e^{-t^2/2}=q^{q/2}e^{-q/2}$.
\end{proof}

\begin{rem}
The last statement asserts that, for $a^2\geq q$, the $q$-th moments of the truncated variables
$\min(|g|,a)$ and $|g|\chi_{\{|g|\leq a\}}$ are equivalent, with a constant multiple,
to the (not truncated) absolute moment $\E|g|^q$ (see \ref{eq:gaussian moments}).
\end{rem}

\subsection{Chernoff--type bounds for order statistics}\label{ss:chernoff}

Given any number $\alpha\in[0,1)$, the {\it quantile of order $\alpha$} with respect to the distribution of $|g|$
is the number $\xi_\alpha$ satisfying $\Prob\{|g|\leq \xi_\alpha\}=\alpha$.
It follows from \eqref{eq:normal approx} that
\begin{equation}
\label{eq:Feller quantile}
\xi_\alpha^{-1}\exp(-\xi_\alpha^2/2)\simeq 1-\alpha,\quad \alpha\geq 1/2.
\end{equation}
Standard estimates for quantiles of the Gaussian distribution
(see, for example, \cite[p.~264]{David}) imply that for $1\leq i\leq n/2$ we have
\begin{equation}
\label{eq:quantile}
\Big|\xi_{1-i/n}-\sqrt{2\log (n/i)}+\frac{\frac{1}{2}\log(\log (n/i))}{\sqrt{2\log (n/i)}}\Big|\lesssim\frac{1}{\sqrt{\log (n/i)}}.
\end{equation}

Further, for the standard Gaussian vector $G=(g_1,g_2,\dots,g_n)$ in $\R^n$, the order statistics of $G$,
denoted by $g_1^*,g_2^*,\dots,g_n^*$,
are the non-increasing rearrangement of the vector of absolute values $(|g_1|,|g_2|,\dots,|g_n|)$.
Given $\beta\in(0,1)$, we have
\begin{align*}
\Prob\{g_i^*\leq \xi_{1-\beta}\}=\sum_{j=0}^{i-1} {n\choose j}\Prob\{|g|\geq \xi_{1-\beta}\}^j\,\Prob\{|g|\leq \xi_{1-\beta}\}^{n-j}
=\sum_{j=0}^{i-1} {n\choose j}\beta^j (1-\beta)^{n-j}.
\end{align*}
It follows from Chernoff's theorem for the partial binomial sums (see \cite{Chernoff} for the original result,
or \cite[p.~24]{Boucheron} as a modern reference) that for $i\leq \beta n$ we have
$$\Prob\{g_i^*\leq \xi_{1-\beta}\}\leq\exp\Big((i-1)\log\frac{\beta n}{i-1}+(n-i+1)\log\frac{n-\beta n}{n-i+1}\Big).$$
Applying the relation $\log(1+t)\leq t-t^2/(2+2t)$ ($t\geq 0$), we get
\begin{align}\label{eq:chernoff leq}
\Prob\{g_i^*\leq \xi_{1-\beta}\}\leq\exp\Big(-\frac{(\beta n-i+1)^2}{2\beta n}\Big),\quad 1\leq i\leq \beta n.
\end{align}

The relation~\eqref{eq:chernoff leq} allows to derive deviation inequalities
for order statistics. Let us remark at this point
that, although order statistics are systematically studied in literature (see classical book \cite{David}, or paper
\cite{BT} as an example of recent developments), we were not able to locate results
in a form convenient for us. For completeness, we provide proofs of next three lemmas.

\begin{lemma}[Lower deviation for large order statistics]
\label{l:lower dev initial}
There are universal constants $C,c>0$ with the following property.
Assume that $n$ is large, and that $1\leq i\leq \sqrt{n}$. Let $\frac{1}{\sqrt{\log n}}\leq u\leq 1-\frac{C}{\log n}$. Then
$$\Prob\big\{g_i^*\leq u\,\xi_{1-i/n}\big\}
\leq\exp\bigg(-\frac{c\,i}{u}\Big(\frac{n}{i\,\sqrt{\log n}}\Big)^{1-u^2}\bigg).$$
\end{lemma}
\begin{proof}
Let $n,i,u$ satisfy the assumptions
and let $s\in(0,1-i/n)$ be such that $\xi_s=u\,\xi_{1-i/n}$.
Observe that, in view of the lower bound on $u$
and the approximation formula \eqref{eq:quantile}, we have $\xi_s\gtrsim 1$.
Then, applying \eqref{eq:Feller quantile} twice, we get
$$
1-s\simeq \frac{1}{\xi_s}\exp\big(-{\xi_s}^2/2\big)
=\frac{{\xi^{u^2-1}_{1-i/n}}}{u\,{\xi^{u^2}_{1-i/n}} }\exp\big(-u^2{\xi^2_{1-i/n}}/2\big)
\simeq \frac{{\xi^{u^2-1}_{1-i/n}}}{u} \bigg(\frac{i}{n}\bigg)^{u^2},
$$
where, by \eqref{eq:quantile}, we have $\xi_{1-i/n}\simeq \sqrt{\log(n/i)}$.
Thus,
\begin{equation}\label{eq: aux p3iuhef}
1-s\simeq \frac{\log(n/i)^{\frac{u^2-1}{2}}}{u} \bigg(\frac{i}{n}\bigg)^{u^2}
=\frac{1}{u}\bigg(\frac{n}{i\sqrt{\log(n/i)}}\bigg)^{1-u^2}\frac{i}{n}.
\end{equation}
The assumptions $u\leq 1-\frac{C}{\log n}$ and $i\le\sqrt{n}$ imply that
$$
\Big(\frac{n}{i\sqrt{\log(n/i)}}\Big)^{1-u^2}\ge \exp\Big(\frac{C}{\log n}\log \frac{n}{i\sqrt{\log(n/i)}}\Big)\ge
\exp\Big({C}-C\frac{\log\log n}{\log n}\Big),
$$
which is bigger than a large absolute constant if $C$ is large enough, whence
$(1-s)n\gg i$.
Applying \eqref{eq:chernoff leq}, we get
$$\Prob\big\{g_i^*\leq u\xi_{1-i/n}\big\}=\Prob\big\{g_i^*\leq \xi_{s}\big\}
\leq \exp\Big(-\frac{((1-s) n-i+1)^2}{2(1-s)n}\Big)
\leq\exp\Big(-\frac{1}{8}(1-s)n\Big).$$
It remains to reuse \eqref{eq: aux p3iuhef}.
\end{proof}

\begin{lemma}[Lower deviation for intermediate order statistics]\label{l:lower dev intermediate}
There is a universal constant $c>0$ with the following property. Let $n$ be large, let $i\leq n/2$
and $u\in(0,1)$. Then
$$\Prob\big\{g_i^*\leq u\xi_{1-i/n}\big\}\leq \exp\Big(- c\,(1-u)^2\,i\,\log\frac{n}{i}\Big).$$
\end{lemma}
\begin{proof}
As in the proof of the above lemma, we let $s\in(0,1-i/n)$ be such that $\xi_s=u\,\xi_{1-i/n}$.
Denoting by $F$ the cdf of $|g|$, we have
$$s\leq 1-\frac{i}{n}-F'(\xi_{1-i/n})(\xi_{1-i/n}-\xi_s)=1-\frac{i}{n}-\sqrt{\frac{2}{\pi}}(1-u)\exp({-\xi^2_{1-i/n}}/2)\,\xi_{1-i/n},$$
whence, applying \eqref{eq:Feller quantile} and \eqref{eq:quantile},
$$1-s-\frac{i}{n}\gtrsim (1-u)\exp({-\xi^2_{1-i/n}}/2)\,\xi_{1-i/n}\gtrsim (1-u)\frac{i}{n}\log\frac{n}{i},$$
and
$$s\leq 1-\frac{i}{n}-\widetilde c\,(1-u)\frac{i}{n}\log\frac{n}{i}$$
for a sufficiently small universal constant $\widetilde c>0$.
Finally, in view of \eqref{eq:chernoff leq},
$$\Prob\big\{g_i^*\leq \xi_{s}\big\}
\leq \exp\Big(-\frac{((1-s) n-i+1)^2}{2(1-s)n}\Big)
\leq \exp\Big(-c\,(1-u)^2\,i\,\log\frac{n}{i}\Big).$$
\end{proof}

The two lemmas above need to be complemented with the following crude bound
for probability of very large deviations.
\begin{lemma}\label{l:lower dev large}
Let $u\geq 0$ and $i\leq n/2$. Then
$$\Prob\big\{g_i^*\leq u\big\}\leq (4u)^{n/2}.$$
\end{lemma}
\begin{proof}
We have
$$\Prob\big\{g_i^*\leq u\big\}\leq {n\choose i}\Prob\{|g|\leq u\}^{n-i}\leq 2^n\, u^{n-i} \leq (4u)^{n/2}.$$
\end{proof}

\section{Negative truncated moments of $\ell_p^n$--norms}
\label{s:negative}

In this section, we derive upper bounds for expressions of the form
$$\E\,\Big(\sum\nolimits_{i=1}^n \min(|g_i|,T)^q\Big)^{-L},$$
where the numbers $q\geq 1$ and $L>0$ are such that $qL=O(\log n)$,
and $T$ is a truncation level which can take any value in the range $[\xi_{1-1/n},\infty]$.
In particular, for $T=\infty$ the above quantity is the $-Lq$-th moment of the $\|\cdot\|_q$--norm ---
$\E \|G\|_q^{-Lq}$. Negative moments of arbitrary norms were considered
in \cite{KV}, where, in particular, bounds for quantities of the form $(\E\|G\|^{-q})^{1/q}$
were derived for $q$ less than $d(\|\cdot\|)$, the ``lower Dvoretzky dimension'' of a norm $\|\cdot\|$.
In \cite{PVZ}, negative $r$-th moments of $\|\cdot\|_q$--norms were considered in the same context as our note;
however, the relations derived in \cite{PVZ} (see, in particular \cite[Lemma~3.6]{PVZ})
do not extend to the case when both $q$ and $r$ are greater than $\log n$.
Finally, let us mention a recent work \cite{PV 16b} where a strong upper bound on
$(\E\|G\|^{-q})^{1/q}$ was obtained in terms of the positive moment $\E\|G\|$
and the variance $\Var\|G\|$ for any norm in $\R^n$.
On the other hand, applying this result of \cite{PV 16b} would require extra care because of
absence of a truncation level in the statement of \cite{PV 16b}, and the necessity to have
precise lower bounds for $\E\|G\|_q$.
The approach we take here is relatively elementary and based on the Chernoff inequality which we used in Section~\ref{ss:chernoff}.

\medskip

We start with the following small ball probability estimate:
\begin{lemma}\label{l:negative small ball}
Let $n$ be a large integer, $G=(g_1,g_2,\dots,g_n)$ be the standard Gaussian vector,
and let $T\in [\xi_{1-1/n},\infty]$ and $q\geq 1$. Then
for any number $\tau\in(0,1/2)$ we have
\begin{align*}
\Prob&\Big\{\sum\nolimits_{i=1}^n \min(|g_i|,T)^q\leq \tau\sum\nolimits_{i=1}^n \xi_{1-i/n}^q\Big\}\\
&\leq \min\Big(C'\exp\big(-c\,n^{(1-(2\tau)^{2/q})/4}\big),
n\big(4(2\tau)^{1/q}\sqrt{2\log n}\big)^{n/2}\Big),
\end{align*}
where $C',c>0$ are universal constants.
\end{lemma}
\begin{proof}
Obviously,
$$\sum\nolimits_{i=1}^n \min(|g_i|,T)^q\geq \sum\nolimits_{i=1}^{\lfloor n/2\rfloor} \min(g_i^*,T)^q,$$
so that for any $\tau\in(0,1/2)$ we have
\begin{align*}
\Prob&\Big\{\sum\nolimits_{i=1}^n \min(|g_i|,T)^q\leq \tau\sum\nolimits_{i=1}^n \xi_{1-i/n}^q\Big\}\\
&\leq \Prob\Big\{\sum\nolimits_{i=1}^{\lfloor n/2\rfloor} \min(g_i^*,T)^q
\leq 2\tau\sum\nolimits_{i=1}^{\lfloor n/2\rfloor} \xi_{1-i/n}^q\Big\}\\
&\leq\sum_{i=1}^{\lfloor n/2\rfloor}\Prob\big\{\min(g_i^*,T)^q\leq 2\tau\xi_{1-i/n}^q\big\}\\
&=\sum_{i=1}^{\lfloor n/2\rfloor}\Prob\big\{g_i^*\leq (2\tau)^{1/q}\xi_{1-i/n}\big\}.
\end{align*}
First, assume that $(2\tau)^{1/q}\leq 1-\frac{C}{\log n}$, where the constant $C>0$ comes from
Lemma~\ref{l:lower dev initial}.
We will divide the above sum into two parts corresponding to large and ``intermediate'' order statistics.
For every $i\leq \sqrt{n}$,
using the notation $r:={\max((2\tau)^{1/q},(\log n)^{-1/2})}^2$, we get,
in view of Lemmas~\ref{l:lower dev initial} and~\ref{l:lower dev large},
\begin{align*}
\Prob&\big\{g_i^*\leq (2\tau)^{1/q}\xi_{1-i/n}\big\}\\
&\leq\min\bigg(\exp\Big(-c\,i\,\Big(\frac{n}{i\,\sqrt{\log n}}\Big)^{1-r^2}\Big),
\big(4(2\tau)^{1/q}\xi_{1-i/n}\big)^{n/2}\bigg)\\
&\leq\min\Big(\exp\big(-c\,i\,n^{(1-(2\tau)^{2/q})/4}\big),\big(4(2\tau)^{1/q}\sqrt{2\log n}\big)^{n/2}\Big).
\end{align*}
Further, for all $\sqrt{n}<i\leq n/2$ we have, by Lemmas~\ref{l:lower dev intermediate} and~\ref{l:lower dev large},
\begin{align*}
\Prob&\big\{g_i^*\leq (2\tau)^{1/q}\xi_{1-i/n}\big\}\\
&\leq\min\bigg(\exp\Big(- c\,\big(1-(2\tau)^{1/q}\big)^2\,i\,\log\frac{n}{i}\Big),
\big(4(2\tau)^{1/q}\xi_{1-i/n}\big)^{n/2}\bigg)\\
&\leq\min\Big(\exp\big(-\widetilde c\,i\,(\log n)^{-2}\big),\big(4(2\tau)^{1/q}\sqrt{2\log n}\big)^{n/2}\Big).
\end{align*}
Combining the estimates (note that the first term in the first minimum form a geometric sum), we get
\begin{align*}
&\sum_{i=1}^{\lfloor n/2\rfloor}\Prob\big\{g_i^*\leq (2\tau)^{1/q}\xi_{1-i/n}\big\}\\
&\leq \min\Big(\widetilde C\exp\big(-c\,n^{(1-(2\tau)^{2/q})/4}\big)+n\exp\big(-\widetilde c\,\sqrt{n}\,(\log n)^{-2}\big),
n\big(4(2\tau)^{1/q}\sqrt{2\log n}\big)^{n/2}\Big)\\
&\leq\min\Big(C'\exp\big(-c\,n^{(1-(2\tau)^{2/q})/4}\big),
n\big(4(2\tau)^{1/q}\sqrt{2\log n}\big)^{n/2}\Big).
\end{align*}

Finally, observe that for $(2\tau)^{1/q}\geq 1-\frac{C}{\log n}$, we have that
$n^{(1-(2\tau)^{2/q})/4}$ is bounded from above by an absolute constant, so the last estimate is trivially
satisfied as long as $C'$ is chosen sufficiently large.
\end{proof}

As a consequence, we obtain
\begin{prop}
\label{p:negative}
For any $K>0$ there are $n_K,v_K>0$ depending only on $K$ with the following property.
Let $n\geq n_K$, let $q\geq 1$ and $0<L\leq K$ be such that $qL\leq K\log n$,
and let $g_1,g_2,\dots,g_n$ be i.i.d.\ standard Gaussians. Then for any $T\in[\xi_{1-1/n},\infty]$ we have
$$\E\,\Big(\sum\nolimits_{i=1}^n \min(|g_i|,T)^q\Big)^{-L}
\leq v_K\,\Big(\sum\nolimits_{i=1}^n {\xi^q_{1-i/n}}\Big)^{-L}.$$
\end{prop}
\begin{proof}
Fix admissible parameters $K,L,q,T$. We will assume that $n$ is large.
For any integer $m\geq 1$, we have
\begin{align*}
\Prob&\Big\{\Big(\sum\nolimits_{i=1}^n \min(|g_i|,T)^q\Big)^{-L}
\geq 2^{m}\,\Big(\sum\nolimits_{i=1}^n {\xi^q_{1-i/n}}\Big)^{-L}\Big\}\\
&=
\Prob\Big\{\sum\nolimits_{i=1}^n \min(|g_i|,T)^q
\leq 2^{-m/L}\,\sum\nolimits_{i=1}^n {\xi^q_{1-i/n}}\Big\}.
\end{align*}
Applying Lemma~\ref{l:negative small ball}, we obtain for all $m\geq 2L$:
\begin{align}
\Prob&\Big\{\Big(\sum\nolimits_{i=1}^n \min(|g_i|,T)^q\Big)^{-L}
\geq 2^{m}\,\Big(\sum\nolimits_{i=1}^n {\xi^q_{1-i/n}}\Big)^{-L}\Big\}\label{eq:aux976wf}\\
&\leq
\min\Big(C'\exp\big(-c\,n^{(1-2^{-m/(Lq)})/4}\big),
n\big(4\cdot 2^{-m/(2Lq)}\sqrt{2\log n}\big)^{n/2}\Big).\nonumber
\end{align}
In the range $2L\leq m\leq Lq$, we have
$$\exp\big(-c\,n^{(1-2^{-m/(Lq)})/4}\big)
\leq \exp\big(-c\,n^{c''m/(Lq)}\big)\leq\exp\big(-c\,e^{c''m/K}\big)$$
for a sufficiently small universal constant $c''>0$.
In particular, for all such $m$ the probability in \eqref{eq:aux976wf} is bounded from above by
$w_K 4^{-m}$, where $w_K>0$ may only depend on $K$.
Further, for $Lq<m\leq 10Lq\log n$, we have
$$\exp\big(-c\,n^{(1-2^{-m/(Lq)})/4}\big)\leq \exp\big(-c\,n^{1/8}\big)\ll 4^{-10Lq\log n}\leq 4^{-m}.$$
Finally, for $m>10Lq\log n$ the probability in \eqref{eq:aux976wf} is bounded by
$$n\big(4\cdot 2^{-m/(2Lq)}\sqrt{2\log n}\big)^{n/2}\leq n 2^{-mn/(8Lq)}\ll 4^{-m}.$$
Combining the estimates, we get for $h:=(\sum\nolimits_{i=1}^n \min(|g_i|,T)^q)^{-L}$ and 
$\zeta:=(\sum\nolimits_{i=1}^n {\xi^q_{1-i/n}})^{-L}$,
$$
\Prob\big\{h\geq 2^{m}\zeta\big\}\leq \widetilde w_K\,4^{-m},\quad m\geq 2L.
$$
Hence, 
$$
\E\,h=\E\big(h\chi_{\{h\in[0,2^{2L}\zeta]\}}\big)+\sum\nolimits_{m=2L}^\infty\E\big(h\chi_{\{h\in[2^m\zeta,2^{m+1}\zeta]\}}\big)\le
2^{2L}\zeta+\sum\nolimits_{m=2L}^\infty \widetilde w_K2^{-m+1}\zeta,
$$
and the result follows.
\end{proof}

\begin{rem}
\label{r:negative}
Note that for any $1\leq q\leq K\log n$, we have
$$
\sum\nolimits_{i=1}^n {\xi^q_{1-i/n}}\simeq_K n\E\min(|g|,\xi_{1-1/n})^q,
$$
where the symbol ``$\simeq_K$'' means that the quantities are equivalent up to a multiple
depending only on parameter $K>0$.
To see this, observe that for any $i\in\{1,2,\dots,n-1\}$ we have
$\Prob\{\min(|g|,\xi_{1-1/n})\in(\xi_{1-i/n},\xi_{1-(i+1)/n}]\}=\frac{1}{n}$, whence
$$\sum_{i=1}^{n-1}{\xi^q_{1-i/n}}\leq n\E\min(|g|,\xi_{1-1/n})^q\leq {\xi^q_{1-1/n}}+\sum_{i=1}^{n-1}{\xi^q_{1-i/n}}.$$
In remains to apply \eqref{eq:quantile} to compare ${\xi^q_{1-1/n}}$ with the power of the second quantile ${\xi^q_{1-2/n}}$.
\end{rem}

\section{Upper bounds for the variance}\label{s:upper}

In this section we obtain upper bounds for $\Var\|G\|_p$, $p\geq C$.
Before we proceed with the proofs, let us provide some motivation for the strategy
we have taken.
As we mentioned in the introduction, the basic tool for estimating the variance from above is
Talagrand's $L_1-L_2$ bound (Theorem~\ref{t:talagrand}).
In \cite{PVZ}, the theorem was directly applied to the norm $\|\cdot\|_p$, which gives the estimate
$$\Var\|G\|_p\lesssim\sum\limits_{i=1}^n \frac{\E|\partial_i \|G\|_p|^2}
{1+\log\big(\sqrt{\E|\partial_i \|G\|_p|^2}/\E|\partial_i \|G\|_p|\big)},$$
where $\partial_i\|G\|_p$ denotes the $i$-th partial derivative of the norm (viewed as a function in $\R^n$)
evaluated at $G$. An elementary computation then leads to an equivalent inequality
\begin{equation}\label{eq: aux 0q873}
\Var\|G\|_p\lesssim \frac{n}{B}\E\frac{|g_1|^{2p-2}}{\big(\sum_{i=1}^n |g_i|^p\big)^{2-2/p}},
\end{equation}
where $B=1+\log\big(\sqrt{\E|\partial_i \|G\|_p|^2}/\E|\partial_i \|G\|_p|\big)$, and so $B$
can be at most logarithmic in $n$. A natural approach to estimating the expectation in the last
formula would be to remove $g_1$ from the denominator and use independence:
\begin{equation}\label{eq: aux3-ih}
\Var\|G\|_p\lesssim \frac{n}{B}\E|g_1|^{2p-2}\,\E\Big(\sum_{i=2}^n |g_i|^p\Big)^{2/p-2}.
\end{equation}
However, this approach fails for all $p>\log_2 n$: the upper bound for the variance we get
this way is {\it worse} than the bound $\Var\|G\|_p\lesssim 1$ that follows from $1$--Lipschitzness of $\|\cdot\|_p$--norm.
To see this, observe that
$$
\E\Big(\sum_{i=2}^n |g_i|^p\Big)^{2/p-2}\ge \big(\E\sum_{i=2}^n |g_i|^p\big)^{2/p-2},
$$
whence, applying standard estimates for absolute moments of Gaussian variables, we get
that the expression on the right hand side of \eqref{eq: aux3-ih} is at least of order $\frac{2^p}{B}n^{2/p-1}$.

In fact, as we show later, the estimate \eqref{eq: aux3-ih}  is not sharp for all $p>\frac{2\log n}{\log (2e)}$.
Clearly, the problem with the above argument lies in the fact that, for large $p$, the input of the individual
coordinate $|g_1|^p$ to the total sum can be huge, and removing the term from the denominator in \eqref{eq: aux 0q873}
alters the expectation.

As a way to resolve the issue, we will consider truncated Gaussian variables.
Given $p\in [1,\infty)$ and a truncation level $T>0$, we introduce an auxiliary function
\begin{equation}
\label{FM}
f_T(G):=\Big(\sum_{i=1}^n\min(T,\,|g_i|^p)\Big)^{1/p},
\end{equation}
so that
\begin{equation}\label{eq:var<}
\Var\|G\|_p\le 2\E(\|G\|_p-f_T(G))^2+2\Var f_T(G),
\end{equation}
and then treat the two terms on the r.h.s.\ separately (the parameter $p$ shall always be clear from the context).
Determining the right truncation level $T$ (when both terms admit satisfactory upper estimates)
is not straightforward. We prefer to postpone the actual definition of the truncation level,
and consider first some general estimates when $p$ is arbitrary and $T\geq \xi_{1-1/n}$.

We start with $\E(\|G\|_p-f_T(G))^2$.
\begin{lemma}
\label{l:tailsL}
For any large integer $n$, any $1\leq p<\infty$ and any truncation level $T\geq \xi_{1-1/n}$ we have
\begin{equation}
\label{tails}
\E(\|G\|_p-f_T(G))^2\lesssim nT^{-3}\exp(-T^2/2).
\end{equation}
\end{lemma}
\begin{proof}
Define a random set $I=I(G):=\{i\leq n:\, |g_i|>T\}$.
Since for any concave function $h$ in $\R$ and any $t\geq 0$ and $x\geq y$, we have
\begin{equation}
\label{eq:conc}
h(x+t)-h(y+t)\leq h(x)-h(y),
\end{equation}
then, taking $h(r):=r^{1/p}$ and $t:=\sum_{i\notin I}|g_i|^p$, we get
\begin{align*}
\|G\|_p-f_T(G)=\Big(\sum_{i\in I}|g_i|^p+t\Big)^{1/p}-\Big(|I|T^p+t\Big)^{1/p}
\leq \Big(\sum_{i\in I}|g_i|^p\Big)^{1/p}-|I|^{1/p}T.
\end{align*}
For every $m\geq 1$, let $\chi_{\{|I(G)|=m\}}$ be the indicator of the event that exactly $m$
coordinates of $G$ are greater (in absolute value) than $T$.
It follows from the above that for every $m\ge 1$ we have
\begin{align*}
\E\big((\|G\|_p-f_T(G))^2\chi_{\{|I(G)|=m\}}\big)
&\le\E\Big[\Big(\Big(\sum_{i\in I}|g_i|^p\Big)^{1/p}-m^{1/p}T\Big)^2\chi_{\{|I(G)|=m\}}\Big]\\
&\leq{n \choose m}\E\Big[\Big(\Big(\sum_{i=1}^m|g_i|^p\Big)^{1/p}-m^{1/p}T\Big)^2\chi_{\{|g_1|>T,\cdots,|g_m|>T\}}\Big],
\end{align*}
where, in view of \eqref{eq:conc},
\begin{align*}
\Big(&\Big(\sum_{i=1}^m|g_i|^p\Big)^{1/p}-m^{1/p}T\Big)^2\\
&=\Big(\sum_{j=0}^{m-1}\Big[\Big(\sum_{i=1}^{m-j}|g_i|^p+jT^p \Big)^{1/p}
-\Big(\sum_{i=1}^{m-j-1}|g_i|^p+(j+1)T^p \Big)^{1/p}\Big]\Big)^2\\
&\leq\Big(\sum_{j=0}^{m-1}\big[|g_{m-j}|-T\big]\Big)^2\\
&\leq m\sum_{j=0}^{m-1}\big[|g_{m-j}|-T\big]^2.
\end{align*}
Hence,
\begin{align*}
\E\big((\|G\|_p-f_T(G))^2\chi_{\{|I(G)|=m\}}\big)
&\leq {n \choose m}m\sum_{j=0}^{m-1}\E\big(\big[|g_{m-j}|-T\big]^2\chi_{\{|g_1|>T,\cdots,|g_m|>T\}}\big)\\
&={n \choose m}m^2\E\big((|g|-T)^2\chi_{\{|g|>T\}}\big)(\Prob\{|g|>T\})^{m-1}.
\end{align*}
In view of \eqref{eq:normal approx}, we have $\Prob\{|g|>T\}\leq \sqrt{2/\pi}\,T^{-1}\exp(-T^2/2)$, and
\begin{align*}
\sqrt{\pi/2}\,\E\big((|g|-T)^2\chi_{\{|g|>T\}}\big)
&=\int_T^\infty(x-T)^2 e^{-x^2/2}\,dx\\
&=e^{-T^2/2}\int _0^\infty y^2e^{-y^2/2}e^{-Ty}\,dy\\
&\leq e^{-T^2/2}\int _0^\infty y^2e^{-Ty}dy\\
&\lesssim T^{-3}\exp(-T^2/2).
\end{align*}
Summarizing, we get
\begin{align*}
\E(\|G\|_p-f_T(G))^2&=\sum_{m=1}^n\E\big((\|G\|_p-f_T(G))^2\chi_{\{|I(G)|=m\}}\big)\\
&\lesssim  T^{-2}\sum_{m=1}^n {n \choose m}m^2 \Big(\sqrt{\frac{2}{\pi}}\frac{1}{T}\exp(-T^2/2)\Big)^m.
\end{align*}
It is easy to show that for any number $a\in\R$ we have
$$
\sum_{m=1}^n {n \choose m}m^2a^m=an(1+a)^{n-1}+a^2n(n-1)(1+a)^{n-2}.
$$
Since in our case $a=\sqrt{{2}/{\pi}}{T}^{-1}\exp(-T^2/2)$, relation~\eqref{eq:Feller quantile}
implies that $(1+a)^{n-1}\lesssim 1$, and
$$\sum_{m=1}^n {n \choose m}m^2 \Big(\sqrt{\frac{2}{\pi}}\frac{1}{T}\exp(-T^2/2)\Big)^m
\lesssim an\leq n{T}^{-1}\exp(-T^2/2).$$
The result follows.
\end{proof}

As the next step, we consider the variance of $f_T(G)$.
\begin{lemma}
\label{l:varf1}
Let $n$ be a large integer, let $p\in[1,3\log n]$ and let $T\geq \xi_{1-1/n}$.
Then
\begin{equation*}
\Var f_T(G)\lesssim \frac{n^{-1+2/p}}{1+\log A}\cdot\frac{\E(|g|^{2p-2}\chi_{\{|g|\leq T\}})}{(\E\min(\xi_{1-1/n},|g|)^{p})^{2-2/p}},
\end{equation*}
where
\begin{equation}\label{eq: A definition}
A:= \max\bigg(1,\frac{\E(|g|^{2p-2}\chi_{\{|g|\le T\}})}{\big(\E(|g|^{p-1}\chi_{\{|g|\le T\}})\big)^2}
\cdot\frac{\big(n\E\min(\xi_{1-1/n},|g|)^{p}\big)^{2-2/p}}{T^{2p-2}+\big(n\E\min(T,\,|g|)^{p}\big)^{2-2/p}}\bigg).
\end{equation}
\end{lemma}
\begin{proof}
It follows from Theorem~\ref{t:talagrand} that
\begin{equation}
\label{eq:talagr1}
\Var(f_T(G))\lesssim \frac{n\E(|\partial_1 f_T|^2)}{1+\log \big(\E(|\partial_1 f_T|^2)/(\E|\partial_1 f_T|)^2\big)},
\end{equation}
where $|\partial_1 f_T|=f_T^{1-p}|g_1|^{p-1}\chi_{\{|g_1|\le T\}}$. First we estimate the numerator.
By Proposition~\ref{p:negative} applied with a constant parameter $K\ge 6$
to the standard $(n-1)$--dimensional truncated Gaussian vector, we have
\begin{align*}
\E(|\partial_1f_T|^2)&\leq \E\big(|g_1|^{2p-2}\chi_{\{|g_1|\le T\}}\big)\,
\E\Big(\sum_{i=2}^n\min(T,\,|g_i|)^{p}\Big)^{2/p-2}
\\
&\lesssim \E\big(|g|^{2p-2}\chi_{\{|g|\le T\}}\big)\,
\Big(\sum_{i=2}^n \xi_{1-i/(n-1)}^p\Big)^{2/p-2}.
\end{align*}
Next, observe that
$$\sum_{i=2}^n \xi_{1-i/(n-1)}^p\simeq \sum_{i=1}^n \xi_{1-i/n}^p$$
(this can be easily verified using relation \eqref{eq:quantile}).
Then, in view of Remark~\ref{r:negative},
$$\E(|\partial_1f_T|^2)\lesssim \E\big(|g|^{2p-2}\chi_{\{|g|\le T\}}\big)\,
\big(n\E\min(|g|,\xi_{1-1/n})^{p}\big)^{2/p-2}.$$

It remains to estimate from below the denominator in \eqref{eq:talagr1}.
Essentially repeating the above computations, we get
\begin{align*}
\E(|\partial_1f_T|)\lesssim\E\big(|g|^{p-1}\chi_{\{|g|\leq T\}}\big)\,\big(n\E\min(|g|,\,\xi_{1-1/n})^{p}\big)^{1/p-1}.
\end{align*}
Further,
\begin{align*}
    \E(|\partial_1 f|^2)&=\E\frac{|g_1|^{2p-2}\chi_{\{|g_1|\le T\}}}{\big(\sum_{i=1}^n\min(T,\,|g_i|)^{p}\big)^{2-2/p}}
    \\
    &\ge  \E\frac{|g_1|^{2p-2}\chi_{\{|g_1|\le T\}}}{\big(T^{p}+\sum_{i=2}^n\min(T,\,|g_i|)^{p}\big)^{2-2/p}}
    \\
    &\ge  \frac{\E\big(|g_1|^{2p-2}\chi_{\{|g_1|\le T\}}\big)}{2 \M\big((T^{p}+\sum_{i=2}^n\min(T,\,|g_i|)^{p})^{2-2/p}\big)}
    \\
    &\gtrsim\frac{\E\big(|g|^{2p-2}\chi_{\{|g|\le T\}}\big)}{ \big(T^{p}+\E\sum_{i=2}^n\min(T,\,|g_i|)^{p}\big)^{2-2/p}}
    \\
    &\gtrsim\frac{\E\big(|g|^{2p-2}\chi_{\{|g|\le T\}}\big)}
   {T^{2p-2}+ \big(n\E\min(T,\,|g|)^{p}\big)^{2-2/p}},
\end{align*}
and the statement follows.
\end{proof}

For shortness, in what follows  we denote
$$
\xi:=\xi_{1-1/n}.
$$
Note that $\x1\exp(\x1^2/2)\simeq n$ and by \eqref{eq:quantile}
\begin{equation}
\label{eq:quantile1}
\x1=\sqrt{2\log\, n}-\frac{1}{2}\frac{\log\log\, n}{\sqrt{2\log\, n}}+O\Big( \frac{1}{\sqrt{\log\, n}}\Big).
\end{equation}

Let us state a combination of the last two lemmas as a corollary:
\begin{cor}\label{cor:combined}
Let $n$ be a large integer, let $p\in[1,3\log n]$, and let $T\geq \x1$. Then
\begin{equation}\label{eq: in cor combined}
\Var\|G\|_p\lesssim nT^{-3}\exp(-T^2/2)+\frac{n^{-1+2/p}}{1+\log A}
\cdot\frac{\E(|g|^{2p-2}\chi_{\{|g|\leq T\}})}{(\E\min(\x1,|g|)^{p})^{2-2/p}},
\end{equation}
where $A$ is defined by \eqref{eq: A definition}.
\end{cor}
Essentially, our work consists in optimizing the above expression over admissible $T$.
It turns out that taking the truncation level close to
\begin{equation}\label{eq:ideal truncation}
\big(n\E\min(\x1,|g|)^p\big)^{1/p}
\end{equation}
produces optimal upper bounds for the variance.
Observe that the quantity in \eqref{eq:ideal truncation} is greater than $\x1$. Indeed,
$$\E\min(\x1,|g|)^p=\sqrt{\frac{2}{\pi}}\int\limits_0^\infty \min(\x1^p,t^p)e^{-t^2/2}\,dt\geq
\sqrt{\frac{2}{\pi}}\int\limits_{\x1}^\infty \x1^p e^{-t^2/2}\,dt=\frac{\x1^p}{n}.$$
Thus, \eqref{eq:ideal truncation} may serve as an admissible truncation level in \eqref{eq: in cor combined}.
The following estimates are implied by Corollary~\ref{c:moments} and relation \eqref{eq:Feller quantile}.
\begin{lemma}[Estimates for $\E\min(\x1,|g|)^p$]\label{l:approximation}
Let $n$ be a large integer and let $p\geq 1$.
Then
\begin{itemize}
\item For $1\leq p\leq \x1^2$, we have
$$\E\min(\x1,|g|)^p\simeq (p/e)^{p/2}.$$
\item For $p\geq \x1^2$, we have
$$\frac{\x1^{p+2}}{\x1+p-\x1^2}\lesssim n\E\min(\x1,|g|)^p\lesssim\frac{p\,\x1^{p}}{\x1+p-\x1^2}.$$
\end{itemize}
\end{lemma}
While working with expression \eqref{eq:ideal truncation} directly may be complicated,
the above lemma allows somewhat simpler (equivalent) definition. For $p\geq 1$, we define
a truncation level $M$ as follows
\begin{equation}\label{eq: M definition}
M^p=M(p)^p:=\begin{cases}n\big(p/e\big)^{p/2},&\mbox{if }1\leq p\leq\x1^2;\\
\x1^{p}\cdot\frac{p}{\x1+p-\x1^2},&\mbox{if }\x1^2<p.\end{cases}
\end{equation}

In the next statement we collect some simple properties of $M$.
\begin{lemma}
\label{l:M}
Provided that $n$ is sufficiently large, we have:
\begin{itemize}
\item $M\geq \x1$ for all $p\geq 1$;

\item If $1\leq p\leq \frac{2\log n}{\log (2e)}$ then $2p-2\leq M^2$;

\item If $\frac{2\log n}{\log (2e)}\leq p\leq \x1^2$ then $p\leq M^2\leq 2p$;

\item If $\x1^2<p$ then $M^2\leq p^{1+\frac{1}{p}}$.

\end{itemize}
\end{lemma}
\begin{proof}
First, taking into account that
\begin{equation}\label{eq:minM2}
\min_{p\in[1,\infty)} n^{2/p}{p/e}=n^{2/p}{p/e}|_{p=2\log n}=2\log n\geq\x1^2,
\end{equation}
we get $M\geq \x1$ for all $p\geq 1$.
In the range $1\leq p\leq \frac{2\log n}{\log (2e)}$, the assertion trivially follows from~\eqref{eq:quantile1}
and the estimate $M\geq\xi$. For $\frac{2\log n}{\log (2e)}\leq p\leq \x1^2$,
we have $2p\geq n^{2/p}{p}/{e}=M^2$, and as $n^{2/p}>e$, we get $M^2>p$.
In the interval $\x1^2<p$ the statement follows from the definition of $M$.
\end{proof}

\begin{lemma}\label{l:computational}
We have $\exp(-n^{2/p}p/(2e))\le n^{-2}2^p$ for all $p\in[1,2\log n]$.
\end{lemma}
\begin{proof}
It is enough to show that
$$
2\log n\le p\log 2+n^{2/p}p/(2e).
$$
The derivative of the right hand side with respect to $p$ is
$$
\log 2- \frac{n^{2/p}}{2e}\Big(\frac{2\log n}{p}-1\Big),
$$
which is less than zero if and only if $p\leq \frac{2\log n}{\log(2e)}$.
Thus, the minimum of $p\log 2+n^{2/p}p/(2e)$ on $[1,2\log n]$ is attained at $p=\frac{2\log n}{\log(2e)}$,
and at the point the expression is equal to $2\log n$.
\end{proof}

\begin{lemma}[Estimates for $M^{-1}\exp(-M^2/2)$]\label{l: MexpM}
Let $n$ be a large integer, $p\geq 1$, and let $M=M(p)$ be defined as before. Then
\begin{itemize}
\item For $1\leq p\leq \x1^2$, we have
$$M^{-1}\exp(-M^2/2)\simeq n^{-1/p}p^{-1/2}\exp\Big(-\frac{p}{2e}n^{2/p}\Big).$$
\item For $\x1^2<p$, we have
$$M^{-1}\exp(-M^2/2)\simeq\frac{1}{n}\Big(1-\frac{\x1^2-\x1}{p}\Big).$$
\end{itemize}
\end{lemma}
\begin{proof}
For $1\leq p\leq \x1^2$ the statement follows directly from the definition of $M$.
suppose that $\x1^2< p$. Then
$M\simeq \x1$, and, applying \eqref{eq:Feller quantile}, we get
\begin{align*}
M^{-1}\exp(-M^2/2)&\simeq \frac{1}{\x1}\exp\bigg(-\frac{\x1^2}{2}
\Big(\frac{p}{\x1+p-\x1^2}\Big)^{2/p}\bigg)\\
&\simeq\frac{1}{n}\exp\bigg(-\frac{\x1^2}{2}
\Big(\Big(\frac{p}{\x1+p-\x1^2}\Big)^{2/p}-1\Big)\bigg).
\end{align*}
We will use the fact that $z-1=\log z+O((z-1)^2)$ for all $z\geq 1$.
Note that
$$\Big(\Big(\frac{p}{\x1+p-\x1^2}\Big)^{2/p}-1\Big)^2\leq \big(p^{2/p}-1\big)^2\ll\x1^{-2}.$$
Hence, the previous estimate implies
\begin{align*}
M^{-1}\exp(-M^2/2)&\simeq \frac{1}{n}\exp\bigg(-\frac{\x1^2}{2}\,
\log\Big(\frac{p}{\x1+p-\x1^2}\Big)^{2/p}\bigg)\\
&=\frac{1}{n}\Big(\frac{\x1+p-\x1^2}{p}\Big)^{\x1^2/p}\\
&\simeq\frac{1}{n}\Big(1-\frac{\x1^2-\x1}{p}\Big),
\end{align*}
where in the last relation we used the fact that
$t^t\geq e^{-1/e}$ for any $t>0$. This proves the lemma.
\end{proof}
The last lemma obviously provides upper bounds for the first term in \eqref{eq: in cor combined} (for $T=M$).

\begin{lemma}\label{l:2p moment estimate}
Let $n$ be a large integer and let $p\geq 1$.
Then
\begin{itemize}
\item If $1\leq p\leq \frac{2\log n}{\log (2e)}$ then
$$\E(|g|^{2p-2}\chi_{\{|g|\leq M\}})\simeq\bigg(\frac{2p}{e}\bigg)^{p-1};$$
\item If $\frac{2\log n}{\log (2e)}\leq p\leq \x1^2$ then
$$\E(|g|^{2p-2}\chi_{\{|g|\leq M\}})\simeq
\frac{1}{\sqrt{\log n}}\cdot\frac{n^{2}(p/e)^{p}}{\sqrt{\log n}+p-\frac{2\log n}{\log(2e)}}\exp\Big(-\frac{p}{2e}n^{2/p}\Big);$$
\item If $\x1^2<p$ then
$$\frac{\x1^{2p}}{n(\x1+p-\x1^2)}\lesssim
\E(|g|^{2p-2}\chi_{\{|g|\leq M\}})\lesssim\frac{p\x1^{2p-2}}{n(\x1+p-\x1^2)}.$$
\end{itemize}
\end{lemma}
\begin{proof}
First, consider the range $1\leq p\leq \frac{2\log n}{\log (2e)}$. We have
$2p-2\leq n^{2/p}{p}/{e}=M^2$, whence, by Corollary~\ref{c:moments},
$$
\E(|g|^{2p-2}\chi_{\{|g|\leq M\}})\simeq \bigg(\frac{2p-2}{e}\bigg)^{p-1}\simeq \bigg(\frac{2p}{e}\bigg)^{p-1}.
$$

\medskip

Next, assume that $\frac{2\log n}{\log (2e)}\leq p\leq \x1^2$.
Here, $2p\geq n^{2/p}{p}/{e}=M^2$, and in view of Corollary~\ref{c:moments},
$$\E(|g|^{2p}\chi_{\{|g|\leq M\}})\simeq \frac{M^{2p+1}e^{-M^2/2}}{M+2p-M^2}.$$
H\"older's inequality then implies
$$\E(|g|^{2p-2}\chi_{\{|g|\leq M\}})
\leq \big(\E(|g|^{2p}\chi_{\{|g|\leq M\}})\big)^{1-1/p}
\lesssim \frac{M^{2p-1}e^{-M^2/2}}{M+2p-M^2}.$$
On the other hand
$$\E(|g|^{2p-2}\chi_{\{|g|\leq M\}})\geq \frac{1}{M^2}\E(|g|^{2p}\chi_{\{|g|\leq M\}}),$$
whence
$$\E(|g|^{2p-2}\chi_{\{|g|\leq M\}})\simeq \frac{M^{2p-1}e^{-M^2/2}}{M+2p-M^2}.$$
Applying the definition of $M$ and the estimates for $M^{-1}\exp(-M^2/2)$ from Lemma~\ref{l: MexpM}, we get
\begin{align*}
\E(|g|^{2p-2}\chi_{\{|g|\leq M\}})&\simeq \frac{M^{2p}n^{-1/p}p^{-1/2}}{M+2p-M^2}\exp\Big(-\frac{p}{2e}n^{2/p}\Big)\\
&= \frac{1}{\sqrt{p}}\cdot\frac{n^{2-1/p}(p/e)^{p}}{n^{1/p}\sqrt{p/e}+2p-n^{2/p}p/e}\exp\Big(-\frac{p}{2e}n^{2/p}\Big).
\end{align*}
Finally, note that, in the given range for $p$, we have
$$n^{1/p}\sqrt{p/e}+2p-n^{2/p}p/e\simeq n^{1/p}\sqrt{p/e}+p\log(2e)-2\log n\simeq \sqrt{\log n}+p-\frac{2\log n}{\log(2e)}.$$

\medskip

Now, we consider the interval $\x1^2<p$.
In this range we have $2p-2\geq 1.5 M^2$, so
\begin{align*}
\frac{M^{2p-1}}{p}e^{-M^2/2}\simeq
\frac{M^{2p-1}e^{-M^2/2}}{M+2p-2-M^2}
&\lesssim\E(|g|^{2p-2}\chi_{\{|g|\leq M\}})\\
&\lesssim\frac{p M^{2p-3}e^{-M^2/2}}{M+2p-2-M^2}\simeq M^{2p-3}e^{-M^2/2}.
\end{align*}
It remains to apply Lemma~\ref{l: MexpM}.
\end{proof}

\begin{lemma}
\label{l:logA}
Let $n$ be a large integer and let $p\in[1,3\log n]$. Then, with $A$ defined by formula~\eqref{eq: A definition}
with $T=M(p)$, we have $1+\log A\gtrsim p$.
\end{lemma}
\begin{proof}
Since $M\geq\x1$ for all $p\ge 1$, we have in view of Lemma~\ref{l:approximation} and the definition of $M$:
$$
M^{2p-2}\simeq\big(n\E\min(\x1,|g|)^{p}\big)^{2-2/p}\leq\big(n\E\min(M,|g|)^{p}\big)^{2-2/p}.
$$
Further, $\E(|g|^{p-1}\chi_{\{|g|\le M\}})\leq\E\min(M,\,|g|)^{p-1}\leq (\E\min(M,\,|g|)^{p})^{1-1/p}$. This leads to
\begin{align*}
A&\gtrsim \frac{\E(|g|^{2p-2}\chi_{\{|g|\le M\}})}{\big(\E(|g|^{p-1}\chi_{\{|g|\le M\}})\big)^2}
\cdot\frac{M^{2p-2}}{\big(n\E\min(M,\,|g|)^{p}\big)^{2-2/p}}
\\
&\gtrsim \frac{M^{2p-2}\E(|g|^{2p-2}\chi_{\{|g|\le M\}})}{n^{2-2/p}\big(\E\min(M,\,|g|)^{p}\big)^{4-4/p}}=:B.
\end{align*}
If $1\leq p\leq \frac{2\log n}{\log(2e)}$, then $2p-2\le M^2=n^{2/p}p/e$, and by Corollary \ref{c:moments}, we have
$$
\E(|g|^{2p-2}\chi_{\{|g|\le M\}})\simeq (2p/e)^{p-1}\quad\mbox{and}\quad \E\min(M,\,|g|)^{p}\simeq (p/e)^{p/2}.
$$
Hence,
$$
B\gsm \frac{(n^{2/p}p/e)^{p-1}(2p/e)^{p-1}}{n^{2-2/p}(p/e)^{2p-2}}\simeq 2^p.
$$

If $\frac{2\log n}{\log(2e)}\leq p\leq \x1^2$, then $p\le M^2=n^{2/p}p/e\le 2p$, and by Corollary \ref{c:moments} and
Lemma \ref{l:2p moment estimate}, we get
\begin{align*}
B\gsm (p/e)^{1-p}\,\E(|g|^{2p-2}\chi_{\{|g|\le M\}})
\gsm \frac{n^2\exp\big(-\frac{p}{2e}n^{2/p}\big)}{\sqrt{\log n}}.
\end{align*}
In the range under consideration, the minimum of $\exp\big(-\frac{p}{2e}n^{2/p}\big)$
is attained at $p=\frac{2\log n}{\log(2e)}$, whence $B\gtrsim n^{2\log 2/\log(2e)}(\log n)^{-1/2}$, and the statement follows.

Finally, if $\x1^2\le p\le 3\log\,n$, then $p^{1+1/p}\geq M^2$ and $p\simeq \xi^2$. Denote $q:=p^{1+1/p}$.
By H\"older's inequality and Corollary~\ref{c:moments} we have
$$
\E\min(M,\,|g|)^{p}\leq
\big(\E\min(M,\,|g|)^{q}\big)^{p/q}\simeq \bigg(\frac{M^{q+1}e^{-M^2/2}}{M+q-M^2}\bigg)^{p/q}
\lesssim pM^pe^{-M^2/2}.
$$
On the other hand, Lemma~\ref{l:2p moment estimate} gives
$$
\E(|g|^{2p-2}\chi_{\{|g|\leq M\}})\simeq \frac{\x1^{2p}}{n(\xi+p-\xi^2)}\gtrsim \frac{\x1^{2p-2}}{n}.
$$
Together with Lemma~\ref{l: MexpM} the estimates imply
\begin{align*}
B\gtrsim  \frac{M^{2p-2}}{n^{3}}\frac{\xi^{2p-2}}{\big(pM^p\exp(-M^2/2)\big)^{4-4/p}}
\gtrsim\frac{n}{p^4}\frac{\xi^{2p-2}}{M^{2p+2}}\gtrsim \frac{n}{p^8}.
\end{align*}
This completes the proof of the lemma.
\end{proof}

A combination of Lemmas~\ref{l: MexpM},~\ref{l:2p moment estimate}, and~\ref{l:logA} with
Corollary~\ref{cor:combined} gives
\begin{prop}
\label{p:varupper}
Let $n$ be a large integer and let $p\in[1,3\log n]$. Then
\begin{itemize}
\item For $1\leq p\leq \frac{2\log n}{\log(2e)}$ we have
$$
\Var\|G\|_p\lesssim \frac{2^p}{p}n^{-1+2/p};
$$
\item For $\frac{2\log n}{\log(2e)}\leq p\leq\x1^2$ we have
$$
\Var\|G\|_p\lesssim \frac{1}{\sqrt{\log n}}\cdot
\frac{n\exp(-\frac{p}{2e}n^{2/p})}{\sqrt{\log n}+p-\frac{2\log n}{\log(2e)}};
$$
\item For $\x1^2<p\leq 3\log n$ we have
$$\Var\|G\|_p\lesssim \frac{1}{\log n}\Big(1-\frac{\x1^2-\x1}{p}\Big).$$
\end{itemize}
\end{prop}
\begin{proof}
First, assume that $1\leq p\leq \frac{2\log n}{\log(2e)}$. By Corollary~\ref{cor:combined},
Lemmas~\ref{l: MexpM},~\ref{l:2p moment estimate}, and~\ref{l:logA} and Corollary~\ref{c:moments} we have
\begin{align*}
\Var\|G\|_p&\lesssim nM^{-3}\exp(-M^2/2)+\frac{n^{-1+2/p}}{p}
\cdot\frac{\E(|g|^{2p-2}\chi_{\{|g|\leq M\}})}{(\E\min(\x1,|g|)^{p})^{2-2/p}}\\
&\lesssim
n^{1-3/p}p^{-3/2}\exp\Big(-\frac{p}{2e}n^{2/p}\Big)
+\frac{n^{-1+2/p}}{p}\bigg(\frac{2p}{e}\bigg)^{p-1}
\big(\E\min(\x1,|g|)^{p}\big)^{2/p-2}\\
&\lesssim n^{1-3/p}p^{-3/2}\exp\Big(-\frac{p}{2e}n^{2/p}\Big)
+\frac{n^{-1+2/p}}{p}\bigg(\frac{2p}{e}\bigg)^{p-1}
\bigg(\frac{p}{e}\bigg)^{1-p}\\
&\lesssim n^{1-3/p}p^{-3/2}\exp\Big(-\frac{p}{2e}n^{2/p}\Big)
+\frac{2^p\,n^{-1+2/p}}{p}.
\end{align*}
It remains to apply Lemma~\ref{l:computational} to the first term.

\medskip

Next, we treat the case $\frac{2\log n}{\log(2e)}\leq p\leq \x1^2$. Using the same argument as above
and the fact that $n^{1/p}=O(1)$, we obtain
\begin{align*}
\frac{\E(|g|^{2p-2}\chi_{\{|g|\leq M\}})}{(\E\min(\x1,|g|)^{p})^{2-2/p}}
&\lesssim \frac{1}{\sqrt{\log n}}\frac{n^{2}(p/e)^{p}\exp(-\frac{p}{2e}n^{2/p})}{\sqrt{\log n}+p-\frac{2\log n}{\log(2e)}}
\big(\E\min(\x1,|g|)^{p}\big)^{2/p-2}\\
&\lesssim\frac{1}{\sqrt{\log n}}\frac{n^{2}(p/e)^{p}\exp(-\frac{p}{2e}n^{2/p})}{\sqrt{\log n}+p-\frac{2\log n}{\log(2e)}}
\bigg(\frac{p}{e}\bigg)^{1-p}\\
&\simeq\frac{\sqrt{\log n}\,n^{2}\exp(-\frac{p}{2e}n^{2/p})}{\sqrt{\log n}+p-\frac{2\log n}{\log(2e)}}.
\end{align*}
Hence,
\begin{align*}
\Var\|G\|_p&\lesssim nM^{-3}\exp(-M^2/2)+\frac{n^{-1}}{p}
\cdot\frac{\E(|g|^{2p-2}\chi_{\{|g|\leq M\}})}{(\E\min(\x1,|g|)^{p})^{2-2/p}}\\
&\lesssim
\frac{n}{(\log n)^{3/2}}\exp\Big(-\frac{p}{2e}n^{2/p}\Big)
+\frac{\sqrt{\log n}}{p}\cdot
\frac{n\exp(-\frac{p}{2e}n^{2/p})}{\sqrt{\log n}+p-\frac{2\log n}{\log(2e)}}\\
&\lesssim \frac{1}{\sqrt{\log n}}\cdot
\frac{n\exp(-\frac{p}{2e}n^{2/p})}{\sqrt{\log n}+p-\frac{2\log n}{\log(2e)}}.
\end{align*}

\medskip

Finally, we consider the range $\x1^2<p\leq 3\log n$.
We have, in view of Lemma~\ref{l:2p moment estimate}, Corollary~\ref{c:moments} and relation~\eqref{eq:Feller quantile}:
\begin{align*}
\frac{\E(|g|^{2p-2}\chi_{\{|g|\leq M\}})}{(\E\min(\x1,|g|)^{p})^{2-2/p}}
&\lesssim \frac{\x1^{2p}}{n(\x1+p-\x1^2)}
\bigg(\frac{\x1^{p+1}\exp(-\x1^2/2)}{\x1+p-\x1^2}\bigg)^{2/p-2}\\
&\simeq \frac{\x1^{2p}}{n(\x1+p-\x1^2)}
\bigg(\frac{\x1^{p+2}}{n(\x1+p-\x1^2)}\bigg)^{2/p-2}\\
&\simeq \frac{n(\x1+p-\x1^2)}{p}.
\end{align*}
Thus,
\begin{align*}
\Var\|G\|_p&\lesssim nM^{-3}\exp(-M^2/2)+\frac{n^{-1}}{p}
\cdot\frac{\E(|g|^{2p-2}\chi_{\{|g|\leq M\}})}{(\E\min(\x1,|g|)^{p})^{2-2/p}}\\
&\lesssim
\frac{1}{p}\Big(1-\frac{\x1^2-\x1}{p}\Big)+
\frac{\x1+p-\x1^2}{p^2}\\
&\lesssim \frac{1}{\log n}\Big(1-\frac{\x1^2-\x1}{p}\Big).
\end{align*}
\end{proof}

Note that in Proposition \ref{p:varupper} we treat the cae $p<3\log n$.  In the regime $p>3\log n$, we will rely on the following result from \cite{PVZ}:
\begin{lemma}[{\cite[Section~3]{PVZ}}]\label{l:PVZ upper}
We have $\Var\|G\|_p\lesssim\frac{1}{\log n}$ for all $n>1$ and $p\geq 2.01$.
\end{lemma}

\section{Lower bounds for the variance}\label{s:lower}

Let us start with a useful auxiliary result from \cite{PVZ}. We provide a proof for the reader's convenience.
\begin{lemma}[{\cite[Section~3]{PVZ}}]\label{l:PVZ}
Let $p\geq 1$ and $n>1$. Then
$$\Var(\|G\|_p)\geq\frac{n}{2p^2}\E\big((|g_1|^p-|g_1'|^p)(\|G\|_p^p+\|G'\|_p^p)^{1/p-1}\big)^2,$$
where $G=(g_1,g_2,\dots,g_n)$ and $G'=(g_1',g_2',\dots,g_n')$ are independent standard Gaussian vectors in $\R^n$.
\end{lemma}
\begin{proof}
First, clearly $\Var(\|G\|_p)=\frac{1}{2}\E(\|G\|_p-\|G'\|_p)^2$.
Next, it can be checked, using elementary convexity properties, that for any two positive real numbers $a,b$ we have
$$\big|a^{1/p}-b^{1/p}\big|\geq \frac{1}{p}|a-b|\Big(\frac{a+b}{2}\Big)^{1/p-1}.$$
Applying the above inequality for $a:=\|G\|_p^p$ and $b:=\|G'\|_p^p$, we obtain
\begin{align*}
\Var(\|G\|_p)&\geq\frac{1}{2p^2}\E\bigg(\big| \|G\|_p^p-\|G'\|_p^p \big|
\Big(\frac{\|G\|_p^p+\|G'\|_p^p}{2}\Big)^{1/p-1}\bigg)^2\\
&=\frac{2^{2-2/p}}{2p^2}
\E\bigg(\Big(\sum_{i=1}^n (|g_i|^p-|g_i'|^p) \Big)
\big(\|G\|_p^p+\|G'\|_p^p\big)^{1/p-1}\bigg)^2\\
&=\frac{2^{2-2/p}}{2p^2}\sum_{i=1}^n\sum_{j=1}^n
\E\Big((|g_i|^p-|g_i'|^p)(|g_j|^p-|g_j'|^p)
\big(\|G\|_p^p+\|G'\|_p^p\big)^{2/p-2}\Big).
\end{align*}
It is easy to see that, for $i\neq j$, the terms in the above sum are equal zero, whence
$$\Var(\|G\|_p)\geq\frac{2^{2-2/p}}{2p^2}\sum_{i=1}^n
\E\Big((|g_i|^p-|g_i'|^p)^2
\big(\|G\|_p^p+\|G'\|_p^p\big)^{2/p-2}\Big).$$
The result follows.
\end{proof}

As a simple corollary, we obtain the main technical element of the section:
\begin{lemma}\label{l:varlower}
There is a universal constant $C$ with the following property.
Assume that $n>1$ and $p\geq C$. Further, let $T\geq 2$ and $\tau\in(0,1)$ be any numbers such that
$$
\sum_{i=1}^{2n-2}|g_i|^p\leq T^p\;\;\mbox{with probability at least}\;\;\tau,
$$
where $g_1,g_2,\dots,g_{2n-2}$ are i.i.d.\ standard Gaussians. Then for the standard Gaussian vector $G$ in $\R^n$ we have
$$
\Var(\|G\|_p)\gtrsim \frac{\tau\, n}{p^2}\frac{\E\big(|g|^{2p}\chi_{\{|g|\leq T\}}\big)}{T^{2p-2}}
\geq \frac{\tau\, n}{p^2}\frac{\big(\E\big(|g|^{2p-2}\chi_{\{|g|\leq T\}}\big)\big)^{\frac{p}{p-1}}}{T^{2p-2}}.
$$
\end{lemma}
\begin{proof}
In view of Lemma~\ref{l:PVZ}, we have
$$
\Var(\|G\|_p)\geq\frac{n}{2p^2}\E\big((|g_1|^p-|g_1'|^p)(\|G\|_p^p+\|G'\|_p^p)^{1/p-1}\big)^2,
$$
where $G'=(g_1',g_2',\dots,g_n')$ is an independent copy of $G=(g_1,g_2,\dots,g_n)$.
By the assumptions on $T$ we have $\Prob\{\sum_{i=2}^n |g_i|^p+\sum_{i=2}^n |g_i'|^p\leq T^p\}\geq \tau$, whence
$$
\E\big((|g_1|^p-|g_1'|^p)(\|G\|_p^p+\|G'\|_p^p)^{1/p-1}\big)^2
\geq \tau\,\E\big((|g_1|^p-|g_1'|^p)(|g_1|^p+|g_1'|^p+T^p)^{1/p-1}\big)^2.
$$
Further, observe that for any two numbers $a\ge 0$ and $0\le b\le 1$ we have $(a-b)^2> a^2/4-1/2$, whence,
in particular,
$$(|g_1|^p-|g_1'|^p)^2\chi_{\{|g_1'|\leq 1\}}
\geq \frac{1}{4}|g_1|^{2p}\chi_{\{|g_1'|\leq 1\}}-\frac{1}{2}\chi_{\{|g_1'|\leq 1\}}.$$
Together with the above inequalities, it gives
\begin{align*}
\Var(\|G\|_p)&\geq\frac{\tau\,n}{2p^2}\E\big((|g_1|^p-|g_1'|^p)(|g_1|^p+|g_1'|^p+T^p)^{1/p-1}\big)^2
\\
&\gtrsim \frac{\tau\,n}{p^2}\E\bigg(\frac{(|g_1|^p-|g_1'|^p)^2\chi_{\{|g_1|\leq T;|g'_1|\leq 1\}}}{(3T^p)^{2-2/p}}\bigg)
\\
&\gtrsim \frac{\tau\,n}{p^2}\frac{\E(|g_1|^{2p}\chi_{\{|g_1|\leq T;|g'_1|\leq 1\}}-{2})}{T^{2p-2}}
\\
&\gtrsim\frac{\tau\,n}{p^2}\frac{\E\big(|g_1|^{2p}\chi_{\{|g_1|\leq T\}}\big)}{T^{2p-2}},
\end{align*}
where in the last step we used that, by Corollary~\ref{c:moments},
$$
\E(|g_1|^{2p}\chi_{\{|g_1|\leq T\}})\ge\E(|g_1|^{2p}\chi_{\{|g_1|\leq 2\}})\gsm p^{-1}{2^{2p+1}}\gg 2
$$
if $p$ is big enough.
\end{proof}

Naturally, we would like to apply the above lemma with $T$ close to $M(p)$ where the truncation level $M(p)$ was
defined in Section~\ref{s:upper}. We have
\begin{lemma}\label{l:varlower adapted}
Let $n$ be a large integer and let $C$ be the constant from Lemma~\ref{l:varlower}.
Then for any $p\geq C$ we have
$$
\Var(\|G\|_p)\gtrsim\frac{n}{p^2}\frac{\big(\E\big(|g|^{2p-2}\chi_{\{|g|\leq M\}}\big)\big)^{\frac{p}{p-1}}}{M^{2p-2}},
$$
where $M=M(p)$ is defined by formula \eqref{eq: M definition}.
\end{lemma}
\begin{proof}
As it was observed back in Lemma~\ref{l:approximation}, we have
$$M^p\gtrsim n\E\min(\x1,|g|)^p.$$
In particular, there is a universal constant $C_1>0$ such that
$$C_1 M^p\geq 4e^3n\E\min(\x1,|g|)^p.$$
By Markov's inequality, given $2n-2$ i.i.d.\ Gaussian variables $g_1,g_2,\dots,g_{2n-2}$,
we have
\begin{align*}
\Prob\Big\{\sum_{i=1}^{2n-2}|g_i|^p> C_1M^p\Big\}
&\leq \Prob\Big\{\sum_{i=1}^{2n-2}\min(\x1,|g_i|)^p> \frac{C_1}{2}M^p\Big\}\\
&\hspace{3cm}+\Prob\Big\{\sum_{i=1}^{2n-2}|g_i|^p\chi_{\{|g_i|\geq \x1\}}> \frac{C_1}{2}M^p\Big\}\\
&\leq e^{-3}+\Prob\Big\{\sum_{i=1}^{2n-2}|g_i|^p\chi_{\{|g_i|\geq \x1\}}> \frac{C_1}{2}M^p\Big\}.
\end{align*}
Further, we observe that
\begin{align*}
\Prob\Big\{\sum_{i=1}^{2n-2}|g_i|^p\chi_{\{|g_i|\geq \x1\}}> \frac{C_1}{2}M^p\Big\}
\leq \prod_{i=1}^{2n-2}\Prob\{|g_i|\geq \x1\}
=1-\bigg(1-\frac{1}{n}\bigg)^{2n-2}\leq 1-e^{-2},
\end{align*}
whence
$$\Prob\Big\{\sum_{i=1}^{2n-2}|g_i|^p> C_1M^p\Big\}\leq 1+e^{-3}-e^{-2}.$$
Thus, $T:=C_1^{1/p}M$ and $\tau:=e^{-2}-e^{-3}$ satisfy conditions of Lemma~\ref{l:varlower}.
Applying Lemma~\ref{l:varlower}, we obtain
$$\Var\|G\|_p
\gtrsim \frac{n}{p^2}\frac{\big(\E\big(|g|^{2p-2}\chi_{\{|g|\leq T\}}\big)\big)^{\frac{p}{p-1}}}{C_1^{2-2/p}M^{2p-2}}
\gtrsim \frac{n}{p^2}\frac{\big(\E\big(|g|^{2p-2}\chi_{\{|g|\leq M\}}\big)\big)^{\frac{p}{p-1}}}{M^{2p-2}}.$$
\end{proof}

As a corollary of Lemma~\ref{l:varlower adapted} and
bounds on truncated moments from Lemma~\ref{l:2p moment estimate}, we obtain
\begin{prop}\label{p:varlower}
Let $n$ be a large integer and let $p\geq C$, where $C$ is defined in Lemma~\ref{l:varlower}. Then
\begin{itemize}
\item For $1\leq p\leq \frac{2\log n}{\log(2e)}$ we have
$$\Var\|G\|_p\gtrsim \frac{2^p}{p}\,n^{-1+2/p};$$
\item For $\frac{2\log n}{\log(2e)}\leq p\leq\x1^2$ we have
$$\Var\|G\|_p\gtrsim
\frac{n\exp(-\frac{p}{2e}n^{2/p})}{\sqrt{\log n}(\sqrt{\log n}+p-\frac{2\log n}{\log(2e)})};$$
\item For $\x1^2<p$ we have
$$\Var\|G\|_p\gtrsim \frac{\x1^4}{p^3}\Big(1-\frac{\x1^2-\x1}{p}\Big).$$
\end{itemize}
\end{prop}
\begin{proof}
First, assume that $1\leq p\leq \frac{2\log n}{\log(2e)}$. Then a combination of Lemmas~\ref{l:varlower adapted}
and~\ref{l:2p moment estimate} and the definition of $M$ gives
\begin{align*}
\Var\|G\|_p\gtrsim \frac{n}{p^2}\frac{\big(\E\big(|g|^{2p-2}\chi_{\{|g|\leq M\}}\big)\big)^{p/(p-1)}}{M^{2p-2}}
\gtrsim \frac{n}{p^2}\frac{\big((2p/e)^{p-1}\big)^{p/(p-1)}}{n^{2-2/p}(p/e)^{p-1}}\simeq\frac{2^p}{p}\,n^{-1+2/p}.
\end{align*}

\medskip

Next, assume that $\frac{2\log n}{\log(2e)}\leq p\leq \x1^2$. Again, combining Lemmas~\ref{l:varlower adapted}
and~\ref{l:2p moment estimate}, we get
\begin{align*}
\Var\|G\|_p&\gtrsim \frac{n}{p^2}\frac{\big(\frac{1}{\sqrt{\log n}}\cdot
\frac{n^{2}(p/e)^{p}}{\sqrt{\log n}+p-\frac{2\log n}{\log(2e)}}
\exp\big(-\frac{p}{2e}n^{2/p}\big)\big)^{p/(p-1)}}{n^{2-2/p}(p/e)^{p-1}}\\
&\simeq\frac{n\,\exp\big(-\frac{p}{2e}n^{2/p}\big)}{\sqrt{\log n}(\sqrt{\log n}+p-\frac{2\log n}{\log(2e)})}.
\end{align*}

\medskip

Finally, consider the case $\x1^2<p$. We have $M^p=\frac{p\x1^p}{\x1+p-\x1^2}$, and
\begin{align*}
\Var\|G\|_p\gtrsim \frac{n}{p^2}
\frac{\big(\frac{\x1^{2p}}{n(\x1+p-\x1^2)}\big)^{p/(p-1)}}{\big(\frac{p\x1^p}{\x1+p-\x1^2}\big)^{2-2/p}}
\simeq \frac{\x1^4}{p^3}\Big(1-\frac{\x1^2-\x1}{p}\Big).
\end{align*}

The statement follows.
\end{proof}

Note that in the regime $p>\x1^2$ the above estimate gives the right order for $\Var\|G\|_p$ only if $p=O(\log n)$. For  $p\gg \log n$ we will use the following estimate from \cite{PVZ}:
\begin{lemma}[{\cite{PVZ}}]\label{l:PVZlower}
There is a universal constant $C>0$ such that
for $n>1$ and all $p\geq C\log n$ we have
$$\Var(\|G\|_p)\gtrsim \frac{1}{\log n}.$$
\end{lemma}

Together Proposition~\ref{p:varupper}, Lemma~\ref{l:PVZ upper},
Proposition~\ref{p:varlower} and Lemma~\ref{l:PVZlower} imply Theorem~A from the introduction in the regime $p\geq C$.
For $1\leq p\leq C$, we refer to \cite{PVZ}.

\section{Proof of Corollary~B}\label{s:corB}

Let $0<\varepsilon,\delta<1$ be given. It follows from Theorem A that there exists  $v_\delta>0$ depending on $\delta$ such that for all sufficiently large $n$, we have
$$
\Var\|G\|_{(2-\delta)\log n}\leq n^{-v_\delta}.
$$
Let $w_\delta:=v_\delta/40$ and $1<k<w_\delta\log n/\log (2/\varepsilon)$. Construct  an $n\times k$ Gaussian matrix $\G$ whose columns are jointly independent standard Gaussian vectors $G_1,G_2,\dots,G_k$ in $\R^n$. Then a uniform random
$k$-dimensional subspace can be defined as
$$
E=\mathrm{span}\,\{G_1,\dots,G_k\}=\{\G x:\,x\in \R^k\}.
$$
 The subspace $E$ is $(1+\varepsilon)$--spherical in $\ell_{p}^n$ for $p:=(2-\delta)\log n$ if
$$
\frac{\sup_{y\in E}\|y\|_p/\|y\|_2}{\inf_{y\in E}\|y\|_p/\|y\|_2}=
\frac{\sup_{x\in S^{k-1}}\|\G x\|_p/\|\G x\|_2}{\inf_{x\in S^{k-1}}\|\G x\|_p/\|\G x\|_2}\leq 1+\varepsilon.
$$
The last inequality holds whenever for all $x\in S^{k-1}$ we have
$$
\big|\|\G x\|_p/\|\G x\|_2-\E(\|\G x\|_p/\|\G x\|_2)\big|\leq\varepsilon' \E(\|\G x\|_p/\|\G x\|_2),
\quad \varepsilon':=\varepsilon/(2+\varepsilon).
$$
Let $a:=\E\|G\|_p$. Note that if $x\in S^{k-1}$ then $\G x$  is a standard Gaussian vector and we have
 $$
\Prob\big\{\big|\|\G x\|_p-a\big|>\varepsilon' a/4\big\}
<\frac{16\Var{\|G\|_p}}{{\varepsilon'}^2a^2}<\frac{16\Var{\|G\|_p}}{{\varepsilon'} ^2}.
$$
Let $\mathcal{N}$ be an $\varepsilon'/5$-net of minimal cardinality in $\ell_2$-metric in $S^{k-1}$.  We have  $|\mathcal{N}|<(15/\varepsilon')^k$ and
 $$
\Prob\big\{\exists y\in\mathcal{N}:\,\big|\|\G y\|_p-a\big|>\varepsilon' a/4\big\}<\Big(\frac{15}{\varepsilon'}\Big)^k\,\frac{16\Var{\|G\|_p}}{{\varepsilon'}^2}.
$$
Conditioning on the event that $\big|\|\G y\|_p-a\big|\le\varepsilon' a/4$ for all $y\in\mathcal{N}$ we have
$$
\|\G\|_{\ell_2\rightarrow\ell_p}\le \frac{(1+\varepsilon'/4) a}{1-\varepsilon'/5}
$$
(see e.g. Lemma 3.2 of \cite{PV 17}).  Thus if there exists $x\in S^{k-1}$ such that $\big|\|\G x\|_p-a\big|>\varepsilon' a/2$ then there exists $y\in\mathcal{N}$ such that $\|x-y\|_2\le\varepsilon'/5$ and
$$
\big|\|\G y\|_p-a\big|\ge\big|\|\G x\|_p-a\big|-\|\G(x-y)\|_p>\varepsilon' a/2- \frac{(1+\varepsilon'/4) a\varepsilon'}{5-\varepsilon'}
>\frac{\varepsilon' a}{4}.
$$
This leads to
\begin{align*}
\Prob\big\{\exists &x\in S^{k-1}:\,\big|\|\G x\|_p-a\big|>\varepsilon' a/2\big\}
<\Big(\frac{15}{\varepsilon'}\Big)^k\,\frac{16\Var{\|G\|_p}}{{\varepsilon'} ^2}.
\end{align*}
  To pass from $\|\G x\|_p$ to $\|\G x\|_p/\|\G x\|_2$, note that given $\widetilde{\varepsilon}<1/3$ and non-negative random variables  $\xi_1$, $\xi_2$, the event that
  $$
  |\xi_1-\E\xi_1|<\widetilde{\varepsilon}\E\xi_1\quad\mbox{and}\quad |\xi_2-\E\xi_2|<\widetilde{\varepsilon}\E\xi_2/3
  $$
  is contained inside the event
   $$
   |\xi_1/\xi_2-\E(\xi_1/\xi_2)|<2\widetilde{\varepsilon}\E(\xi_1/\xi_2).
   $$
  By the standard concentration estimates, we have
   $$
   \Prob\big\{\forall x\in S^{k-1}:\,\big|\|\G x\|_2-\E\|\G x\|_2\big|<t\E\|\G x\|_2\big\}>1-\frac{c}{t^2n},
   \quad \forall t>0
   $$
   hence, taking $\xi_1=\|G\|_p$, $\xi_2=\|G\|_2$, and $\widetilde{\varepsilon}=\varepsilon'/2$, we get
  \begin{align*}
\Prob \big\{&\exists x\in S^{k-1}:\,\big|\|\G x\|_p/\|\G x\|_2-\E(\|\G x\|_p/\|\G x\|_2)\big|>\varepsilon' \E(\|\G x\|_p/\|\G x\|_2)\big\}
\\
&<
\Prob\big\{\exists x\in S^{k-1}:\,\big|\|\G x\|_p-a\big|>\varepsilon' a/2\big\}
\\
&\quad+\Prob\big\{ \exists x\in S^{k-1}:\,\big|\|\G x\|_2-\E\|\G x\|_2\big|>\varepsilon'\E\|\G x\|_2/6\big\}
\\
&<\Big(\frac{15}{\varepsilon'}\Big)^k\,\frac{16\Var{\|G\|_p}}{{\varepsilon'}^2}+\frac{c}{{\varepsilon'}^2n}
\\
&<\Big(\frac{45}{\varepsilon}\Big)^k\frac{48n^{-v_\delta}}{{\varepsilon}^2}+\frac{c}{{\varepsilon}^2n}
\\
&<n^{-w_\delta}
\end{align*}
provided that $n$ is big enough. Thus (\ref{eq:spherical}) is proved.
\medskip

To prove the second part of Corollary~B, it is enough to consider the case $k=2$. Then $E=\mathrm{span}\,\{G,G'\}$, where $G$ and $G'$ are two independent standard Gaussian vectors in $\R^n$,  and
$$
\frac{\sup_{y\in E}\|y\|_p/\|y\|_2}{\inf_{y\in E}\|y\|_p/\|y\|_2}\ge \frac{\|G\|_p/\|G\|_2}{\|G'\|_p/\|G'\|_2}.
$$
Thus it is enough to show that for $p:=(2+\delta)\log n$ we have
$$
\Prob\Big\{\frac{\|G\|_p/\|G\|_2}{\|G'\|_p/\|G'\|_2}
>1+\frac{w_{\delta}}{\log n}\Big\}\geq w_\delta
$$
for some $w_\delta>0$ depending only on $\delta$.
Observe that standard concentration estimates imply
$$\Prob\Big\{\|G'\|_2/\|G\|_2\leq 1-\frac{1}{\log^2 n}\Big\}\ll 1,$$
whence
it is enough to show that
$$
\Prob\Big\{\frac{\|G\|_p}{\|G'\|_p}
>1+\frac{\widetilde w_{\delta}}{\log n}\Big\}\geq \widetilde w_{\delta}
$$
for some $\widetilde w_{\delta}>0$.
Recall that $\Var\|G\|_p\geq \frac{v_\delta}{\log n}$ for some $v_\delta>0$. Hence,
\begin{align*}
\frac{2v_\delta}{\log n}&\leq \E(\|G\|_p-\|G'\|_p)^2\\
&=2\int_0^\infty t\,\Prob\big\{\big|\|G\|_p-\|G'\|_p\big|\geq t\big\}\,dt\\
&\leq \frac{v_\delta}{\log n}+2\int_{\sqrt{v_\delta/\log n}}^\infty t\,\Prob\big\{\big|\|G\|_p-\|G'\|_p\big|\geq t\big\}\,dt.
\end{align*}
Next, observe that $\E\|G\|_p\simeq \sqrt{\log n}$, and in view of $1$-symmetry of the $\|\cdot\|_p$--norm and by a result from
\cite{T 2017}, we have
$$\Prob\big\{\big|\|G\|_p-\|G'\|_p\big|\geq t\big\}\lesssim n^{-c't/\E\|G\|_p},\quad t>0.$$
This, together with the above relation, implies
$$\frac{v_\delta}{\log n}\lesssim \int_{\sqrt{v_\delta/\log n}}^{C/\sqrt{\log n}} t\,\Prob\big\{\big|\|G\|_p-\|G'\|_p\big|\geq t\big\}\,dt,$$
whence
$$\Prob\big\{\big|\|G\|_p-\|G'\|_p\big|\geq \sqrt{v_\delta/\log n}\big\}\gtrsim  v_\delta.$$
This, and the fact that $\|G\|_p\simeq \sqrt{\log n}$
with very large probability, implies the statement.

\end{document}